\documentclass[12pt]{amsart}
\usepackage{amssymb,amscd}
\usepackage{verbatim}
\usepackage{enumerate}

\makeatletter

\let\mc\mathcal
\usepackage{eucal}
\let\euc\mathcal
\let\mathcal\mc

\let\Medskip\medskip
\def\medskip{\par\Medskip}
\let\Bigskip\bigskip
\def\bigskip{\par\Bigskip}

\let\Maketitle\maketitle
\def\maketitle{\Maketitle\thispagestyle{empty}\let\maketitle\empty}

\newtheorem{thm}{Theorem}[section]
\newtheorem{cor}[thm]{Corollary}
\newtheorem{lem}[thm]{Lemma}
\newtheorem{prop}[thm]{Proposition}

\numberwithin{equation}{section}

\theoremstyle{definition}
\newtheorem*{rem}{Remark}
\newtheorem*{example}{Example}

\def\beq{\begin{equation}}
\def\eeq{\end{equation}}
\def\be{\begin{equation*}}
\def\ee{\end{equation*}}

\def\bea{\begin{eqnarray*}}
\def\eea{\end{eqnarray*}}
\def\bean{\begin{eqnarray}}
\def\eean{\end{eqnarray}}
\def\Ref#1{{\rm(\ref{#1})}}

\def\gadv#1{\global\advance#1 1}
\newcount\itemlet
\def\newbi{\itemlet 96} \newbi
\def\bitem{\gadv\itemlet\endgraf\hangindent1.5\parindent
\hglue-.5\parindent\textindent{\upshape\rlap{\char\the\itemlet}\hp{b})}}

\def\Bitem{\gadv\itemlet\endgraf\hangindent2.5\parindent
\hglue.5\parindent\textindent{\upshape\rlap{\char\the\itemlet}\hp{A}.}}

\newcount\itemrm

\def\iitem{\gadv\itemrm\endgraf\hangindent1.5\parindent\hglue-.5\parindent
\textindent{\upshape\hp{v}\llap{\romannumeral\the\itemrm})}}

\let\Bbb\mathbb

\def\mkern{\kern\mathsurround}
\let\math\m@th
\def\$#1${\mkern$\m@th#1$}

\let\bls\baselineskip \let\ignore\ignorespaces
\let\adv\advance
\def\vsk#1>{\vskip#1\bls}
\def\vv#1>{\vadjust{\vsk#1>}\ignore}
\def\vvn#1>{\vadjust{\nobreak\vsk#1>\nobreak}\ignore}
\def\vvv#1>{\vskip\z@\vsk#1>\nt\ignore}
\def\vvgood{\vadjust{\penalty-500}}

\def\Goodbreak{\par\penalty-\@m}
\def\vvGood{\vadjust{\penalty-\@m}}

\def\nnGood{\noalign{\penalty-\@m}}
\def\wwgood#1:#2>{\vv#1>\vvgood\vv#2>\vv0>}
\def\vskgood#1:#2>{\vsk#1>\goodbreak\vsk#2>\vsk0>}
\def\mathbox#1{\hbox{\m@th$#1$}}

\def\>{\relax\ifmmode\mskip.666667\thinmuskip\relax\else\kern.111111em\fi}
\def\+{\relax\ifmmode\mskip.333333\thinmuskip\relax\else\kern.0555556em\fi}
\def\<{\relax\ifmmode\mskip-.333333\thinmuskip\relax\else\kern-.0555556em\fi}
\def\?{\relax\ifmmode\mskip-.666667\thinmuskip\relax\else\kern-.111111em\fi}

\let\dsize\displaystyle \let\tsize\textstyle
\let\ssize\scriptstyle \let\sss\scriptscriptstyle
\def\fratop{\genfrac{}{}{0pt}1}
\def\satop#1#2{\fratop{\ssize#1}{\ssize#2}}

 \let\vp\vphantom \let\hp\hphantom

  \let\nt\noindent
\def\hline{\hbox to\hsize}
\let\cline\centerline \let\lline\leftline \let\rline\rightline
\def\nn#1>{\noalign{\vskip#1\p@}} \def\NN#1>{\openup#1\p@}
 
\let\Lim\lim \def\lim{\Lim\limits} \let\Sum\sum \def\sum{\Sum\limits}
 
\let\Prod\prod \def\prod{\Prod\limits}
\let\Int\int \def\int{\Int\limits}

\let\colon\: \def\:{^{\vp|}}

\def\const{\mathrm{const}}

\let\alb\allowbreak

 \let\x\times 
\let\le\leqslant \let\ge\geqslant
 
 \let\8\infty \let\*\star
\let\bra\langle \let\ket\rangle
\def\Bra{\bigl\bra} \def\Ket{\bigr\ket}

\let\Hat\widehat \let\Tilde\widetilde

\def\lsym#1{#1\alb\ldots\relax#1\alb}
\def\lc{\lsym,}   
\def\llc{\,,\alb\ {\ldots\ ,}\alb\ }

\let\texspace\ \def\ {\ifmmode\alb\fi\texspace}

\def\Line#1{\kern-.5\hsize\hline{\m@th$\dsize#1$}\kern-.5\hsize}
\def\Lline#1{\kern-.5\hsize\lline{\m@th$\dsize#1$}\kern-.5\hsize}
\def\Cline#1{\kern-.5\hsize\cline{\m@th$\dsize#1$}\kern-.5\hsize}
\def\Rline#1{\kern-.5\hsize\rline{\m@th$\dsize#1$}\kern-.5\hsize}

\def\Ll@p#1{\llap{\m@th$#1$}} \def\Rl@p#1{\rlap{\m@th$#1$}}
 \def\Cl@p#1{\llap{\m@th$#1$\hss}}
\def\Llap#1{\mathchoice{\Ll@p{\dsize#1}}{\Ll@p{\tsize#1}}{\Ll@p{\ssize#1}}%
{\Ll@p{\sss#1}}}
\def\Clap#1{\mathchoice{\Cl@p{\dsize#1}}{\Cl@p{\tsize#1}}{\Cl@p{\ssize#1}}%
{\Cl@p{\sss#1}}}
\def\Rlap#1{\mathchoice{\Rl@p{\dsize#1}}{\Rl@p{\tsize#1}}{\Rl@p{\ssize#1}}%
{\Rl@p{\sss#1}}}
 
\def\LRtph#1#2{\setbox\z@\hbox{#1}\dimen\z@\wd\z@\hbox{\hbox to\dimen\z@{#2}}}
\def\LRph#1#2{\LRtph{\m@th$#1$}{\m@th$#2$}}

\def\Lto#1{\setbox\z@\hbox{\m@th$\tsize{#1}$}%
\mathrel{\mathop{\hbox to\wd\z@{\rightarrowfill}}\limits#1}}
\def\Lgets#1{\setbox\z@\hbox{\m@th$\tsize{#1}$}%
\mathrel{\mathop{\hbox to\wd\z@{\leftarrowfill}}\limits#1}}

\def\ftext#1{{\let\thefootnote\relax\footnotetext{\vsk-.8>\nt#1}}}
\def\textindent#1{\indent\llap{#1\enspace}\ignorespaces}

\let\al\alpha
\let\bt\beta

 \let\eps\varepsilon \let\epsilon\eps

\let\la\lambda

 \let\phi\varphi

\def\C{\Bbb C}

\def\Z{\Bbb Z}

\let\bs\boldsymbol

\def\slN{\mathfrak{sl}_N}

\def\Cl{\C^{\>\lb}}
\def\Sl{S_{\+\lb}}

\def\lab{\bs\la}
\def\lb{\bs l}
\def\tb{\bs t}
\def\Tb{\bs T}

\def\Ic{\mc I}

\def\Ae{\euc A}
\def\Be{\euc B}
\def\Fe{\euc F}
\def\Ie{{\euc I}}
\def\Je{{\euc J}}
\def\Te{\euc T}
\def\Ue{\euc U}

\def\wc{\check w}

\def\uh{\hat u}
\def\Uhe{\Hat\Ue}

\def\pt{\tilde p}

\def\ut{\tilde u}
\def\Tt{\Tilde T}
\def\Tte{\Tilde\Te}

\def\Ute{\Tilde\Ue}

\def\wt{\Tilde w}
\def\yt{\tilde y}

\def\Qu#1{Q^{\+#1}}
\def\QT{\Qu\Te}

\def\ts{{\+\tb}}
\def\DU{D_{\+\Ue}}
\def\DUp{D_{\+\Ue'}}
\def\DUpp{D_{\+\Ue''}}
\def\DUh{D_{\+\Uhe}}
\def\DUhp{D_{\+\Uhe'}}
\def\DUt{D_{\+\Ute}}
\def\DUtp{D_{\+\Ute'}}
\def\DUti_#1{D_{\+\Ute_{#1}}}

\def\ij{i,\+j}

\def\il{i\+l}

\def\KZ/{{\slshape KZ\/}}
\def\qKZ/{{\slshape qKZ\/}}
\def\XXX/{{\slshape XXX\/}}
\def\XXZ/{{\slshape XXZ\/}}

\makeatother

\begin{document}

\vp1
\vsk->

\title[\XXZ/\+-type Bethe ansatz equations and quasi\+-polynomials]
{\XXZ/\>-type Bethe ansatz equations and quasi\+-polynomials}

\author[Jian Rong Li \,and \,Vitaly Tarasov]
{Jian Rong Li$^\star$ and \,Vitaly Tarasov$\<^*$}

\maketitle

\begin{center}
{\it\$^\star$Department of Mathematics, Lanzhou University, Lanzhou 730000,
P.\,R.\,China\/}

\vsk.5>
{\it\$\kern-.4em^*\<$Department of Mathematical Sciences,
Indiana University\,--\>Purdue University Indianapolis\kern-.4em\\
402 North Blackford St, Indianapolis, IN 46202-3216, USA\/}

\vsk.5>
{\it\$^*\<$St.\,Petersburg Branch of Steklov Mathematical Institute\\
Fontanka 27, St.\,Petersburg, 191023, Russia\/}
\end{center}

\ftext{\math$^\star\<${\sl E\>-mail}:\enspace lijr@lzu.edu.cn\>,
\>lijr07@gmail.com\\
$^*\<${\sl E\>-mail}:\enspace vt@math.iupui.edu\>, vt@pdmi.ras.ru\\
$^*\+$Supported in part by NSF grant DMS-0901616}

\begin{abstract}
We study solutions of the Bethe ansatz equation for the \XXZ/\+-type
integrable model associated with the Lie algebra $\slN$. We give a
correspondence between solutions of the Bethe ansatz equations and collections
of quasi\+-polynomials. This extends the results of E.\,Mukhin and
A.\>Varchenko for the \XXX/-type model and the trigonometric Gaudin model.
\end{abstract}

\section{Introduction}
\label{intro}

In this paper we study solutions of the Bethe ansatz equation for
the \XXZ/\+-type integrable models associated with the Lie algebra $\slN$,
see \Ref{Be}. These equations arise in the Bethe ansatz method of computing
eigenvalues and eigenvectors of commuting Hamiltonians of integrable models.
The method gives the eigenvalues and eigenvectors by evaluating explicit
rational functions on solutions of the Bethe ansatz equations,
see for instance \cite{KBI} and references therein.

\vsk.2>
Solutions of the Bethe ansatz equations for the Gaudin model, both rational
and trigonometric, as well for the \XXX/-type model have been studied by
E.\,Mukhin and A.\>Varchenko in \cite{MV1}\,--\,\cite{MV3}. They established
a correspondence between solutions of the Bethe ansatz equations and spaces
of polynomials, quasi\+-polynomials or quasi\+-exponentials with certain
properties. In this paper we extend the results of \cite{MV2}, \cite{MV3}
to the case of the \XXZ/-type Bethe ansatz equations.

\vsk.2>
Our method to construct a collection of quasi\+-polynomials corresponding to
a solution of the Bethe ansatz equations is not completely analogous to that of
\cite{MV2}, \cite{MV3}. As a result, we managed to weaken technical assumptions
on solutions of the Bethe ansatz equations compared with those imposed in
\cite{MV3}. Another advantage is that our method works withut restrictions
for the root of unity case.

\vsk.2>
The plan of the paper is as follows. In Section \ref{s-two}, we describe
the \XXZ/-type Bethe ansatz equations and refine the problem. In particular,
we define regular and admissible solutions of the Bethe ansatz equations.
In Section \ref{s-qpol}, we introduce regular collections of
quasi\+-polynomials and show that each such a collection gives a regular
solution of the Bethe ansatz equations. We also formulate there the main result
of the paper, Theorem \ref{main}, which says that every admissible regular
solution of the Bethe ansatz equations comes from a regular collection of
quasi\+-polynomials. We prove Theorem \ref{main} in Section \ref{s-proof}.
In Section \ref{s-diff}, for a collection of quasi\+-polynomials \>$\Ue$
we consider the monic difference operator \>$\DU$ whose kernel is generated
by \>$\Ue$. We show that the operator \>$\DU$ has rational coefficients if and
only if there is a regular collection of quasi\+-polynomials \>$\Ute$ such that
\>$\DUt=\DU$. Also, for each solution \>$\tb$ of the Bethe ansatz equations
we define a difference operator \>$D^\ts\?$ and show that the collection of
quasi\+-polynomials \>$\Ue$ associated with \>$\tb$ gives a basis of the kernel
of \>$D^\ts\?$, that is, \>$D^\ts\?=\DU$. The Appendix contains necessary
technical identities.

\vsk.2>
The authors thank E.\,Mukhin for helpful discussions.

\section{Bethe ansatz equations and quasi\+-polynomials}
\label{s-two}
\subsection{Bethe ansatz equations}
\label{s-bae}
Throughout the paper we fix a complex number \>$q$ such that \>$q\ne0,\pm\+1$.
More precisely, we fix the value of \,$\log\+q$ and for any complex number
$\al$ set \>$q^{\>\al}\?=\exp\>(\al\>\log\+q)$.

\vsk.2>
Fix an integer $N\ge 2$. Given a collection \>$\lb=(\>l_1\lc l_{N-1})$ of
nonnegative integers, consider the space
\,${\Cl\?=\+\C^{\>{l_1}}\!\lsym\oplus\C^{\>l_{N-1}}}\?$. We label coordinates
on this space by two indi\-ces, the superscript enumerating the summands and
the subscript enumerating coordinates within a summand. That is, we write
\vvn.2>
\be
\tb\,=\,(\+t^{(1)}_1\lc t^{(1)}_{l_1}\!\llc
t^{(N-1)}_1\lc t^{(N-1)}_{l_{N-1}}\+)\,\in\,\Cl\>.
\vv.1>
\ee
The product of the symmetric groups
${\Sl\<=S_{l_1}\!\<\lsym\x S_{l_{N-1}}}\!$ acts on $\Cl\!$ by permuting
coordinates with the same superscript. The \$\Sl$\+-orbit of a single point
will be called an {\it elementary \$\Sl$-orbit\/}.

\vsk.2>
Let $T_1(x)\lc T_{N-1}(x)$ be monic polynomials in one variable and
$\la_1\lc\la_N\<$ be complex numbers. Consider a system of equations
on the variables \>$t^{(a)}_b\!$, \,$a=1\lc\+N-1$, \>$b=1\lc\>l_a$,
\vvn-.4>
\begin{align}
\label{Be}
T_i(t_j^{(i)}q^2)\,\prod_{r=1}^{l_{i-1}}\,(t_j^{(i)}q-t_r^{(i-1)}q^{-1})\;
\prod_{s=1}^{l_{i+1}}\,(t_j^{(i)}\!-t_s^{(i+1)})\;
\prod_{\satop{k=1}{k\ne j}}^{l_i}\,({}& t_j^{(i)}q^{-1}\!-t_k^{(i)}q)\,=\,{}
\\[1pt]
{}={}\,q^{\>2\+(\la_{i+1}-\la_i+\>l_i)-\>l_{i-1}-\>l_{i+1}}\;
T_i(t_j^{(i)})\,\prod_{r=1}^{l_{i-1}}\,(t_j^{(i)}\!-t_r^{(i-1)})\;
\prod_{s=1}^{l_{i+1}}\,(t_j^{(i)}q^{-1}\!-{}& t_s^{(i+1)}q)\,
\prod_{\satop{k=1}{k\ne j}}^{l_i}\,(t_j^{(i)}q-t_k^{(i)}q^{-1})\,,
\notag
\end{align}
$i=1\lc N-1$, $j=1\lc l_i$. Here and later we assume that \,$l_0\<=\+l_N\<=0$,
unless otherwise stated. Equations \Ref{Be} are called the \XXZ/\+-type Bethe
ansatz equations associated with the data \,$\lb=(l_1\lc l_{N-1})$,
$\Tb=(T_1\lc T_{N-1})$, $\lab=(\la_1\lc\la_N)$.

\vsk.2>
The group $\Sl\?$ acts on solutions of equations \Ref{Be} --- \>for every
solution \>$\tb$ of \Ref{Be}, all points of the \$\Sl$\+-orbit
of \>$\tb$ are solutions of \Ref{Be} as well.

\vsk.2>
Equations \Ref{Be} arise in the Bethe ansatz method for the quantum integrable
model defined on the irreducible finite-dimensional representation
of the quantum affine algebra $U_q(\Hat{\slN})$. The polynomials
$T_1(x),T_2(x\+q^{-2})\lc T_{N-1}(x\+q^{-2\+(N\<-1)})$ are the Drinfeld
polynomials of the representation and the parameters
\>$q^{\>2\la_1}\!\lc\>q^{\>2\la_N}\!$ describe the algebra of commuting
Hamiltonians of the model. Notice that in the literature equations \Ref{Be}
are usually written in the form
\enlargethispage{.6\bls}
\begin{align}
\label{Bae}
\frac{T_i(t_j^{(i)}q^2)}{T_i(t_j^{(i)})}\,\,
\prod_{r=1}^{l_{i-1}}\,\frac{t_j^{(i)}q-t_r^{(i-1)}q^{-1}}
{t_j^{(i)}\!-t_r^{(i-1)}}\,\,
\prod_{s=1}^{l_{i+1}}\,\frac{t_j^{(i)}\!-t_s^{(i+1)}}
{t_j^{(i)}q^{-1}\!-t_s^{(i+1)}q}\,\,
& \prod_{\satop{k=1}{k\ne j}}^{l_i}\,\frac{t_j^{(i)}q^{-1}\!-t_k^{(i)}q}
{t_j^{(i)}q-t_k^{(i)}q^{-1}}\;={}
\\[6pt]
& \;{}=\,q^{\>2\+(\la_{i+1}-\la_i+\>l_i)-\>l_{i-1}-\>l_{i+1}}\,,
\notag
\\[-24pt]
\notag
\end{align}
\,$i=1\lc N-1$, $j=1\lc l_i$.

\subsection{Regular solutions of Bethe ansatz equations}
\label{s-rs}
For \>$i=1\lc N-1$, define the polynomials
\vvn-.7>
\beq
\label{p}
p_i(x\+;\tb\+)\,=\,\prod_{j=1}^{l_i}\,(x-t_j^{(i)})
\vv-.8>
\eeq
and
\begin{align}
\label{P}
P_i(x\+;\tb\+)\,={} &\,q^{\>2\+\la_{i+1}}\+p_i(x\+q^{\>2};\tb\+)\,
p_{i-1}(x\+;\tb\+)\,p_{i+1}(x\+q^{-2};\tb\+)\,T_i(x)\>+{}
\\[4pt]
&\?{}+\>q^{\>2\+\la_i}\+p_i(x\+q^{-2};\tb\+)\,p_{i-1}(x\+q^{\>2};\tb\+)\,
p_{i+1}(x\+;\tb\+)\,T_i(x\+q^{\>2})\,.
\notag
\\[-12pt]
\notag
\end{align}
Then equations \Ref{Be} take the form
\vvn.3>
\beq
\label{Bep}
P_i(t^{(i)}_j;\tb\+)\,=\,0\,,\qquad i=1\lc N-1\,,\quad j=1\lc l_i\,.\kern-4em
\vv.3>
\eeq
A solution \>$\tb$ of equations \Ref{Be} is called {\it regular\/} if for any
\>${i=1\lc N-1}$, and any subset \>$I\<\subset\{\+1\lc l_i\}$, the following
condition holds: suppose that all the coordinates \>$t^{(i)}_j\!$,
\>$j\<\in I\<$ are equal with the common value \>$t^{(i)}_{\+I}\?$; then
\vvn.3>
\beq
\label{Bem}
\frac{d^{\>k}}{d\+x^{\+k}}\+P_i(x\+;\tb\+)|_{x=t^{(i)}_{\+I}}\>=\,0\,,\qquad
k=0\lc|\+I\+|-1\,,\kern-2em
\vv.4>
\eeq
where and \>$|\+I\+|$ is the cardinality of $I$. Clearly, if for each
\>${i=1\lc N-1}$, all the coor\-di\-nates \>$t^{(i)}_j\!$, \>$j=1\lc l_i$,
are distinct, the solution \>$\tb$ is regular.

\vsk.2>
If \>$\tb$ is a regular solution of equations \Ref{Be}, then all
points of the \$\Sl$\+-orbit of \>$\tb$ are regular solutions of equations
\Ref{Be} as well.

\vsk.2>
The next proposition is valid by definition of a regular solution of equations
\Ref{Be}.

\begin{prop}
\label{reg}
A point \>$\tb$ is a regular solution of equations \Ref{Be} if and only if
for any \>$i=1\lc N-1$, the polynomial \,$p_i(x\+;\tb\+)\<$ divides the
polynomial \,$P_i(x\+;\tb\+)$.
\end{prop}

\section{Quasi\+-polynomials}
\label{s-qpol}
\subsection{Quasi\+-polynomials}
\label{s-pol}
A quasi\+-polynomial $f(x)$ of type \>$\al$ is an expression of the form
$f(x)=x^{\>\al}p(x,\log x)$, where $\al\in\C$ and $p(x,s)$ is a polynomial.
Call the quasi\+-polynomial $f(x)$ \+log{\it-free\/} if $p(x,s)$ does not
depend on $s$. Say that a polynomial $r(x)$ divides $f(x)$ if $r(x)$ divides
$p(x,s)$. The product of quasi\+-polynomials of types \>$\al$ and \>$\bt$
is supposed to be of type \>$\al+\bt$. Clearly, the product of two nonzero
quasi\+-polynomials is not zero.

\vsk.2>
We will use only algebraic properties of quasi\+-polynomials.
The key relation is: if $f(x)=x^{\>\al}p(x,\log x)$, then
\vvn-.4>
\be
f(x\+q^{\>\bt})\,=\,
q^{\>\al\+\bt}x^{\>\al}p(x\+q^{\>\bt},\log x+\bt\>\log\+q)\,.
\ee

\vsk.3>
Given functions $g_1(x)\lc g_k(x)$, their discrete Wronskian
$W_k[\>g_1\lc g_k]$ is defined by the rule
\vvn-.2>
\beq
\label{Wr}
W_k[\>g_1\lc g_k](x)\,=\,
\det\>\bigl(g_i(x\+q^{-2(j-1)})\bigr)_{\ij=1}^k.
\eeq

\vsk.2>
Let $\lab=(\la_1\lc \la_N)$ be a collection of complex numbers.
By definition, a {\it collection of quasi\+-polynomials of type} $\lab$
is a sequence of quasi\+-polynomials \>$\Ue=(u_1\lc u_N)$ such that for
every \>$i=1\lc N\?$, the quasi\+-polynomial \>$u_i$ has type \>$\la_i$,
and \>$W_{N\>}[\>u_1\lc u_N\+]\ne 0$. We call the collection \>$\Ue$
{\it semiregular\/} if \>$W_{N\>}[\>u_1\lc u_N\+]$ is \>log\+-free.

\vsk.2>
For a semiregular collection of quasi\+-polynomials \>$\Ue=(u_1\lc u_N)$ of type
\vvn.16>
$\lab$, a sequence of monic polynomials \>$\Te=(T_1\lc T_N)$ is called
a {\it preframe\/} of \>$\Ue$ if for any $k=1\lc N-1$ and any subset
$\{\+i_1\lc i_k\}\subset\{1\lc N\+\}$, the product
$\prod_{i=1}^k\,\prod_{j=0}^{k-i}\,T_{N\<-\+i\++1}(x\+q^{-2j})$
divides the quasi\+-polynomial $W_k[\>u_{i_1}\lc u_{i_k}](x)$ of
type $\la_{i_1}\!\lsym+\+\la_{i_k}$, and
\beq
\label{WN}
W_{N\>}[\>u_1\lc u_N\+](x)\,=\,\const\,\,x^{\+\la_1\lsym+\>\la_N}\,
\prod_{i=1}^N\,\prod_{j=0}^{N\?-i}\,T_{N\<-\+i\++1}(x\+q^{-2j})\,.
\vv-.7>
\eeq
Denote
\vvn-.6>
\beq
\label{QT}
\QT_k(x)\,=\,
\prod_{i=1}^k\,\prod_{j=0}^{k-i}\;\<T_{N\<-\+i\++1}(x\+q^{-2j})\,,
\qquad k=1\lc N\>.\kern-2em
\vv-.2>
\eeq
Say that a preframe \>$\Te$ is stronger than a preframe
\>$\Tte=(\Tt_1\lc\Tt_N)$ if \>$\Qu{\Tte}_k(x)$ divides \>$\QT_k(x)$
for any $k=1\lc N-1$.

\vsk.2>
Until the end of this section we fix $\lab=(\la_1\lc \la_N)$ and assume that
every collection of quasi\+-polynomials is of type $\lab$.

\begin{lem}
\label{lemframe}
Let \,\>$\Ue=(u_1\lc u_N)$ be a semiregular collection of quasi\+-polynomials.
There is a preframe \,$\Te=(T_1\lc T_N)$ of \;$\Ue$ such that for every
\>$k=1\lc N\!$, the polynomial \,$\QT_k(x)$, see \>\Ref{QT},
is the greatest common divisor of the quasi\+-polynomials
\;$W_k[\>u_{i_1}\lc u_{i_k}](x)$, where $\{\+i_1\lc i_k\}$
runs over all \>\$k$-element subsets of $\{1\lc N\+\}$ and
\;$W_k[\>u_{i_1}\lc u_{i_k}](x)$ has type
$\la_{i_1}\!\lsym+\+\la_{i_k}$.
\end{lem}
\begin{proof}
The proof is similar to that of Lemma~4.9 in \cite{MV2}.
\end{proof}

The preframe \>$\Te$ defined in Lemma \ref{lemframe} is clearly the strongest
preframe of \>$\Ue$. It is called the {\it frame} of \>$\Ue$.
\vsk.2>
For a semiregular collection of quasi\+-polynomials \>$\Ue=(u_1\lc u_N)$ and
a preframe \>$\Te=(T_1\lc T_N)$ of \>$\Ue$, define quasi\+-polynomials
\,$y_0\lc y_{N-1}$ by the rule \,$u_N\<=y_{N-1}\,\QT_1$,
\vvn.3>
\beq
\label{defy}
W_{N\<-\+i\>}[\>u_{i+1}\lc u_N\+](x)\,=\,y_i(x)\,\QT_{N-i}(x)\,,
\qquad i=0\lc N-2\,,\kern-2em
\vv.3>
\eeq
where \>$\QT_1\lc\QT_N\?$ are given by \Ref{QT}. Notice that $y_0(x)$ is
proportional to $x^{\+\la_1\lsym+\>\la_N}\<$. Call the semiregular collection
\>$\Ue$ {\it regular\/} if the quasi\+-polynomials \>$y_1\lc y_{N-1}$ are
\+log\+-free.

\vsk.2>
Assume that the collection \>$\Ue$ is regular. For $i=1\lc N-1$, let \,$l_i$
be the degree of the polynomial $x^{-\la_{i+1}\lsym-\>\la_N}\>y_i(x)$, and
\,$t^{(i)}_1\!\lc t^{(i)}_{l_i}$ be its roots. That is,
\vvn.2>
\beq
\label{deft}
y_i(x)\,=\,c_i\,x^{\+\la_{i+1}\lsym+\>\la_N}\,
\prod_{j=1}^{l_i}\,(x-t_j^{(i)})\,,\qquad i=1\lc N-1\,,\kern-2em
\eeq
for some nonzero complex numbers \>$c_1\lc c_{N-1}$.
Set \,$\lb=\>\lb_{\>\Ue,\Te}=(\>l_1\lc l_{N-1})$,
and denote by \>$X_{\Ue,\Te}$ the \$\Sl$\+-orbit of the point
\vvn.1>
$(\+t^{(i)}_j)_{\>i=1\lc\+N-1,\;j=1\lc\>l_i}$ in $\Cl\<$.

\begin{thm}
\label{thmV}
Let \,$\Ue\<$ be a regular collection of quasi\+-polynomials of
type $\lab$, \>$\Te=(T_1\lc T_N)$ be a preframe of \,$\Ue$, and
\,$\lb=\>\lb_{\>\Ue,\Te}=(\>l_1\lc l_{N-1})$. Then \,$X_{\Ue,\Te}$ is
an elementary \$\Sl$-orbit of regular solutions of equations \Ref{Be}
associated with the data \,$l_1\lc l_{N-1}$, $T_1\lc T_{N-1}$, $\la_1\lc\la_N$.
\end{thm}
\begin{proof}
Define the quasi\+-polynomials \,$\yt_1\lc\yt_{N-1}$ by the rule
\vvn.3>
\beq
\label{defyt}
W_{N\<-\+i\>}[\>u_i,u_{i+2}\lc u_N\+](x)\,=\,\yt_i(x)\,\QT_{N-i}(x)\,.\!\!
\vv.3>
\eeq
By Lemmas~\ref{common}, \ref{Wid}, and formula~\Ref{QT}, we have
\vvn.3>
\beq
\label{Wyy}
W_{2\>}[\>\yt_i,y_i](x)\,=\,y_{i-1}(x)\,y_{i+1}(xq^{-2})\,T_i(x)\,,
\eeq
that is,
\vv-.2>
\be
\yt_i(x)\,y_i(x\+q^{-2})-\yt_i(x\+q^{-2})\,y_i(x)\,=\,
y_{i-1}(x)\,y_{i+1}(x\+q^{-2})\,T_i(x)\,.
\vv-.2>
\ee
Therefore
\begin{align}
\label{yyy}
& y_i(x)\>\bigl(\yt_i(x\+q^{\>2})\,y_i(x\+q^{-2})-
\yt_i(x\+q^{-2})\,y_i(x\+q^{\>2})\bigr)\,={}
\\[4pt]
& \!\!{}=\,y_i(x\+q^{\>2})\,y_{i-1}(x)\,y_{i+1}(x\+q^{-2})\,T_i(x)+
y_i(x\+q^{-2})\,y_{i-1}(x\+q^{\>2})\,y_{i+1}(x)\,T_i(x\+q^{\>2})\,.
\notag
\\[-16pt]
\notag
\end{align}
If the point \>$\tb=(t^{(i)}_j)$ is defined by \Ref{defy}, \Ref{deft},
the polynomials \>$p_i(x\+;\tb\+)$, \>$P_i(x\+;\tb\+)$ are given by \Ref{p},
\Ref{P}, and the polynomial \>$\pt_i(x)$ equals
\>$x^{-\la_i-\la_{i+2}\lsym-\>\la_N}\yt(x)$, then relation \Ref{yyy}
is equivalent to
\vvn.1>
\be
P_i(x\+;\tb\+)\,=\,p_i(x)\>
\bigl(q^{\>2\+(\la_i-\la_{i+1})}\>\pt_i(x\+q^{\>2})\,p_i(x\+q^{-2})-
q^{\>2\+(\la_{i+1}-\la_i)}\>\pt_i(x\+q^{-2})\,p_i(x\+q^{\>2})\bigr)\,.
\vv.2>
\ee
By Proposition \ref{reg}, this proves Theorem \ref{thmV}.
\end{proof}

\begin{rem}
By Lemma~\ref{common}, a collection of quasi\+-polynomials \>$\Ue=(u_1\lc u_N)$
has a preframe $\Te=(T_1\lc T_N)$ if and only if \>$\Ute=(u_1/T_N\lc u_N/T_N)$
is a collection of quasi\+-polynomials and has a preframe
$\Tte=(T_1\lc T_{N-1},\>1)$. Clearly, \>$X_{\Ue,\Te}=X_{\Ute,\Tte}$, and
for both \>$\Ue$ and \>$\Ute$ equations \Ref{Be} are the same since they
involve only the polynomials $T_1\lc T_{N-1}$. Thus discussing relations
between collections of quasi\+-polynomials and solutions of Bethe ansatz
equations \Ref{Be} we can restrict ourselves without loss of generality
to preframes of the form $(T_1\lc T_{N-1},\>1)$.
\end{rem}

Say that a point \>$\tb=(t^{(i)}_j)$ is {\it generic\/} with respect to
the polynomials \>$\Tb=(T_1\lc T_{N-1})$ if \>$t_j^{(i)}\<\ne\>t_k^{(i+1)}\?$
for all \,$i=1\lc N-2$, \,$j=1\lc l_i$, \,$k=1\lc l_{i+1}$, and
\>$T_i(t^{(i)}_j)\ne0$ for all \,$i=1\lc N-1$, \,$j=1\lc l_i$. Clearly,
all points of the \$\Sl$\+-orbit of a generic point \>$\tb$ are generic.

\begin{lem}
\label{genframe}
Let \,$\Ue$ be a regular collection of quasi\+-polynomials
of type $\lab$, \>$\Te=(T_1\lc T_N)$ be a preframe of \,$\Ue$, and
\,$\lb=\>\lb_{\>\Ue,\Te}=(\>l_1\lc l_{N-1})$. Assume that \,$X_{\Ue,\Te}\?$
is the \$\Sl$-orbit of a generic point with respect to the polynomials
$T_1\lc T_{N-1}$. Then \,$\Te\<$ is the frame of \,$\Ue$.
\end{lem}
\begin{proof}
Let \>$p_i(x)=x^{-\la_{i+1}\lsym-\>\la_N}\>y_i(x)$.
Suppose \>$\Tte=(\Tt_1\lc\Tt_N)$ is a preframe of \,$\Ue$ strictly stronger
than \>$\Te$. Let \>$i$ be the largest number such that \>$\Tt_i\ne T_i$.
Notice that \>$i>1$. There is a number \>$a$ such that $\Tt_i(a)=0$ \>and
$T_i(a)\ne 0$. This implies that \>$p_{i-1}(a)=0$ \>and
\>$p_{i-2}(a)\,T_{i-1}(a)=0$, see \Ref{QT}, \Ref{defy}. However, if
$X_{\Ue,\Te}\<$ is the \$\Sl$-orbit of a generic point, then for every
$j=1,\lc N-1$, the polynomials \>$p_{j-1}\,T_j$ and \>$p_j$ are coprime.
The claim follows.
\end{proof}

\subsection{Main result}
\label{s-thm}
A point \>$\tb=(t^{(i)}_j)_{\>i=1\lc\+N-1,\;j=1\lc\>l_i}$ is called
{\it admissible\/} if
\beq
\label{cond1}
t^{(i)}_j\<\ne\>t^{(i)}_kq^2\>,\qquad i=1\lc N-1\,,\quad j\+,k=1\lc l_i\,.
\vv-.1>
\eeq
In particular, \>$t^{(i)}_j\<\ne 0$ for all \,$i=1\lc N-1$, \,$j=1\lc l_i$.
Clearly, all points
of the \$\Sl$\+-orbit of an admissible point \>$\tb$ are admissible.

\vsk.2>
Fix collections of complex numbers $\lab=(\la_1\lc\la_N)$, nonnegative integers
\,$\lb=(l_1\lc l_{N-1})$ and monic polynomials \>$\Tb=(T_1\lc T_{N-1})$.
The following theorem is the main result of the paper.

\begin{thm}
\label{main}
Let \,$\tb$ be an admissible regular solution of equations \>\Ref{Be}
associated with the data \,$\lb\+,\+\Tb\<,\+\lab$. Then there is a regular
collection of quasi\+-polynomials \,$\Ue$ of type $\lab$ such that
\,$\Te=(T_1\lc T_{N-1},\>1)$ is a preframe of \;$\Ue$ and the \$\Sl$-orbit
of \;$\tb$ equals \>$X_{\Ue,\Te}$.
\end{thm}

Theorem \ref{main} will be proved in Section \ref{s-proof}. The proof is going
in three steps. First we prove the theorem for $N=2$. The obtained statement
is employed then at the second step to construct the required collection of
quasi\+-polynomials \>$\Ue$ for general $N\?$. The final step is to show that
\,$\Te=(T_1\lc T_{N-1},\>1)$ is a preframe of the constructed collection
\,$\Ue$.

\begin{cor}
\label{main2}
Let \,$\tb$ be a generic admissible regular solution of equations \>\Ref{Be}
associated with the data \,$\lb\+,\+\Tb\<,\+\lab$. Then there is a regular
collection of quasi\+-polynomials \,$\Ue$ of type $\lab$ such that
\,$\Te=(T_1\lc T_{N-1},\>1)$ is the frame of \;\<$\Ue$ and the \$\Sl$-orbit
of \;$\tb$ equals \>$X_{\Ue,\Te}$.
\end{cor}
\begin{proof}
The statement follows from Theorem \ref{main} and Lemma \ref{genframe}.
\end{proof}

\section{Proof of Theorem \ref{main}}
\label{s-proof}
\subsection{Proof of Theorem \ref{main} for $N=2$}
\label{s-N2}
Let \>$y(x)=x^{\>\al}p(x)$ be a \+log\+-free quasi\+-polynomial of type
\>$\al$. We call \>$y(x)$ admissible if the polynomials \>$p(x)$ and
\>$p(x\+q^{-2})$ are coprime. In particular, this implies that \>$p(0)\ne 0$.
Thus for an admissible quasi\+-polynomial $y(x)$, the number \>$\al$ and
the polynomial $p(x)$ are determined uniquely.

\vsk.2>
Let \>$y(x)=x^{\>\al}p(x)$ be an admissible quasi\+-polynomial of type \>$\al$.
Since \>$p(x)$ and \>$p(x\+q^{-2})$ are coprime, it is known that there are
unique polynomials \>$r(x)$ and \>$s(x)$ of degree at most \>$\deg\+p$ \>such
that \>$r(x)\>p(x)+s(x)\>p(x\+q^{-2})=1$. Define the quasi\+-polynomials
\>$\Ae[\+y\+]$ and \>$\Be[\+y\+]$ of type \>$-\+\al$ by the rule
\vvn-.1>
\beq
\label{ABy}
\Ae[\+y\+](x)\,=\,x^{-\al}r(x)\,,\qquad \Be[\+y\+](x)\,=\,x^{-\al}s(x)\,,
\vv-.2>
\eeq
so that
\vvn-.2>
\beq
\label{AB1}
y(x)\,\Ae[\+y\+](x)+y(x\+q^{-2})\,\Be[\+y\+](x)\,=\,1\,.
\eeq

\vsk.4>
For a polynomial \>$P(s)$ and a number \>$c$,
let \>$\Ic\+[\+P,c\>]$ be the unique polynomial such that
\vvn.3>
\be
\Ic\+[\+P,c\>](s)-c\>\>\Ic\+[\+P,c\>](s-2\+\log\+q)\,=\,P(s)\,,
\vv.2>
\ee
$\deg\>\Ic\+[\+P,c\>]=\deg\+P$ \>if \>$c\ne1$, \>and
\,$\deg\>\Ic\+[\+P,1\+]=1+\deg\+P$, \;$\Ic\+[\+P,1\+](0)=0$.
\vvn.1>
For example,
\be
\Ic\+[\>1\>,\<1\>](s)\,=\,\frac s{2\+\log\+q}\;,\qquad
\Ic\+[\+1\>,c\>](s)\,=\,\frac1{1-c}\;,\quad c\ne 1\,.
\vv.3>
\ee
For a quasi\+-polynomial $f(x)=x^{\>\al}\>\sum_i\>x^i\+P_i\+(\+\log x)$ of type
\vvn.1>
\>$\al$, define the quasi\+-polynomial $\Ie\+[\+f\+]$ of type \>$\al$ by the rule
\,$\Ie\+[\+f\+](x)=
x^{\>\al}\>\sum_i\>x^i\,\Ic\+[\+P_i\+,q^{-2\al}\+](\+\log x)$, so that
\be
\Ie\+[\+f\+](x)-\Ie\+[\+f\+](x\+q^{-2})\,=\,f(x)\,.
\vv-.5>
\ee
For example,
\vvn-.3>
\be
\Ie[\>1\>]\,=\,\frac{\log x}{\log\+q^{\>2}}\,,\qquad
\Ie\+[\+x^{\>\al}\+]\,=\,\frac{x^{\>\al}}{1-q^{-2\al}}\,,\quad
q^{-2\al}\<\ne\+1\,.
\ee

\vsk.3>
For a quasi\+-polynomial $f(x)=x^{\>\al}p(x,\log x)$ of type \>$\al$ and
a \+log\+-free quasi\+-polynomial $g(x)=x^{\>\bt}r(x)$ of type \>$\bt$, define
the quasi\+-polynomial $\bra\+f(x)/g(x)\ket_+\?=x^{\>\al-\bt}\+h(x,\log x)$
of type \>$\al-\bt$ by requiring that $h(x,s)$ is the polynomial part of
the ratio $p(x,s)/r(x)$, that is,
\,$\deg_{\>x}\bigl(p(x,s)-r(x)\>h(x,s)\bigr)<\deg\+r(x)$.
If $f(x)=g(x)\>[f(x)/g(x)]_+$, we say that $g(x)$ divides $f(x)$.

\vsk.2>
For an admissible quasi\+-polynomial \>$y(x)=x^{\>\al}p(x)$ of type \>$\al$
and a quasi\+-polynomial \>$V(x)$ of type \>$\bt$, define the quasi\+-polynomial
\>$\Fe[\>y,\<V\+]$ of type \>$\bt-\al$ as follows. Let \>$a=\Ae[\+y\+]$ and
\>$b=\Be\+[y\+]$, see \Ref{ABy}. Consider the quasi\+-polynomial
\vvn.3>
\be
v(x)\,=\,\biggl\bra\+
\frac{a(x)\>V(x)+b(x\+q^{-2})\>V(x\+q^{-2})}{y(x\+q^{-2})}\+\biggr\ket_{\!\!+}
\vv.4>
\ee
of type \>$\bt-2\+\al$. Set \,$\Je[\>y,\<V\+]=\Ie[\+v\+]$ \>and
\vvn.4>
\beq
\label{Fx}
\Fe[\>y,\<V\+](x)\,=\,V(x)\,\Be[\+y\+](x)+y(x)\,\Je[\>y,\<V\+](x)\,.
\vvn.4>
\eeq

\begin{prop}
\label{Vy}
Let \,$y(x)$ be an admissible quasi\+-polynomial of type \>$\al$, \>$V(x)\<$
be a quasi\+-polynomial of type \>$\bt$. Let \>$Y\?=\Fe\+[\>y,\<V\+]$.
Assume that \,$y(x)\<$ divides the quasi\+-polynomial
\,$y(x\+q^{\>2})\>V(x)+y(x\+q^{-2})\>V(x\+q^{\>2})$ of type \>$\al+\bt$.
Then \,$W_2[\>Y\?,y\>]=V\?$.
\end{prop}
\begin{proof}
Let \>$a=\Ae[\+y\+]$, \,$b=\Be[\+y\+]$, \,$f=\Je[\>y,\<V]$. By \Ref{AB1},
\vvn.2>
\beq
\label{ab}
a(x)\>y(x)+b(x)\>y(x\+q^{-2})\,=\,1
\eeq
and
\vvn-.5>
\begin{align*}
& a(x\+q^{\>2})\>b(x)\>\bigl(y(x\+q^{\>2})\>V(x)+
y(x\+q^{-2})\>V(x\+q^{\>2})\bigr)\,={}
\\[3pt]
&\,{}=\,a(x\+q^{\>2})\>V(x\+q^{\>2})+b(x)\>V(x)-\bigl(
a(x)\>a(x\+q^{\>2})\>V(x\+q^{\>2})+b(x)\>b(x\+q^{\>2})\>V(x)\bigr)\>y(x)\,,
\\[-14pt]
\end{align*}
so \>$y(x)$ divides the quasi\+-polynomial
\>$a(x\+q^{\>2})\>V(x\+q^{\>2})+b(x)\>V(x)$ of type \>$\bt-\al$. Hence
\vvn.3>
\be
y(x\+q^{-2})\>\bigl(f(x)-f(x\+q^{-2})\bigr)\,=\,
a(x)\>V(x)+b(x\+q^{-2})\>V(x\+q^{-2})\,,
\vv.3>
\ee
and \,$W_2[\>Y\?,y\>]=V$ \>by an easy simplification using \Ref{ab}.
\end{proof}

\begin{proof}[Proof of Theorem \ref{main} for $N=2$]
Let \,$\tb=(\+t^{(i)}_j)$ be an admissible regular solution of
equations \>\Ref{Be} and $p_1(x\+;\tb\+)$ be the polynomial given by \Ref{p}.
\vvn.06>
Define an admissible quasi\+-poly\-nomial \>$y(x)=x^{\+\la_2}p_1(x\+;\tb\+)$ of
type $\al=\la_2$ and a quasi\+-polynomial \>$V(x)=x^{\+\la_1+\la_2}\>T_1(x)$
of type \>$\bt=\la_1+\la_2$.

\vsk.2>
Let \>$u_2\<=y$ and \>$u_1\<=\Fe[\>y,\<V\+]$.
Proposition \ref{reg} shows that \>$y(x)$ divides the quasi\+-poly\-nomial
\,$y(x\+q^{\>2})\>V(x)+y(x\+q^{-2})\>V(x\+q^{\>2})$ of type \>$\al+\bt$.
Hence by Proposition \ref{Vy},
\vvn.3>
\be
W_2[\>u_1,u_2]\,=\,x^{\+\la_1+\la_2}\>T_1(x)\,,
\vv.2>
\ee
that is, \>$\Ue=(u_1,u_2)$ is the required collection of quasi\+-polynomials
of type $\lab=(\la_1,\la_2)$.
\end{proof}

\subsection{Construction of a collection of quasi\+-polynomials}
\label{s-Ue}
Let \,$\tb=(\+t^{(i)}_j)$ be an admis\-sible regular solution of equations
\>\Ref{Be} associated with the data \,$\lb\+,\+\Tb\<,\+\lab$. \>Let
\vvn.2>
\beq
\label{px}
p_i(x)\,=\,\prod_{j=1}^{l_i}\,(x-t_j^{(i)})\,,\qquad i=1\lc N-1\,,\kern-2em
\vv.1>
\eeq
cf.~\Ref{p}, and \>$p_0(x)=p_N(x)=1$. Then for every \>$i=0\lc N\?$,
\vvn.3>
\be
y_i(x)\,=\,x^{\+\la_{i+1}\lsym+\>\la_N}\+p_i(x)
\vv.2>
\ee
is an admissible quasi\+-polynomial of type $\la_{i+1}\<\lsym+\la_N$.
Set \>$a_i=\Ae[\+y_i]$ \>and \>$b_i=\Be[\+y_i]$, see \Ref{ABy}, so that
\beq
\label{abi}
a_i(x)\>y_i(x)+b_i(x)\>y_i(x\+q^{-2})\>=\>1\,.
\vv.3>
\eeq

\vsk.2>
The next lemma is equivalent to Proposition \ref{reg}.

\begin{lem}
\label{div}
For any \,$i=1\lc N-1$, the quasi\+-polynomial \,$y_i(x)\<$ divides
the quasi\+-poly\-nomial
\beq
\label{Ai}
A_i(x)\,=\,y_i(x\+q^{\>2})\>y_{i-1}(x)\>y_{i+1}(x\+q^{-2})\>T_i(x)+
y_i(x\+q^{-2})\>y_{i-1}(x\+q^{\>2})\>y_{i+1}(x)\>T_i(x\+q^{\>2})
\vv.2>
\eeq
of type \>$\la_i+2\+\la_{i+1}+3\>(\la_{i+2}\<\lsym+\la_N)$.
\end{lem}

Set \>$Q_1(x)=1$ \>and
\vvn-.5>
\beq
\label{Qk}
Q_k(x)\,=\,
\prod_{i=2}^k\,\prod_{j=0}^{k-i}\;\<T_{N\<-\+i\++1}(x\+q^{-2j})\,,
\qquad k=2\lc N\>,\kern-2em
\eeq
cf.~\Ref{QT}.
We will construct a collection of quasi\+-polynomials \>$\Ue=(u_1\lc u_N)$
of type $\lab$ such that \>$u_N\<=y_{N-1}$ and
\vvn.2>
\beq
\label{Wuu}
W_{N\<-\+i\>}[\>u_{i+1}\lc u_N\+](x)\,=\,y_i(x)\,Q_{N-i}(x)\,,
\qquad i=0\lc N-2\,,\kern-3em
\vv.4>
\eeq
cf.~\Ref{defy}. We employ the recursive procedure described below. Given the
quasi\+-polynomials \>$u_{i+1}\lc u_N\?$, we obtain the quasi\+-polynomial \>$u_i$
by formula \Ref{ui} and verify relation \Ref{Wuu} in Proposition \ref{A}.

\vsk.2>
For the first step of the process, set
\vvn.06>
\,$u_{N-1}=\>\Fe[\>y_{N-1},y_{N-2}\>T_{N-1}\+]$.
Then Lemma \ref{div} for $i=N-1$ and Proposition \ref{Vy} yield
\vvn.3>
\beq
\label{W-1}
W_2[\>u_{N-1},u_N\+]\,=\,W_2[\>u_{N-1},y_{N-1}]\,=\,y_{N-2}\>T_{N-1}\>=
\,y_{N-2}\,Q_2\,,
\vv.3>
\eeq
which is relation \Ref{Wuu} for $i=N-2$.

\begin{lem}
\label{div2}
The quasi\+-polynomial \,$y_{N-2}(x)\<$ divides the quasi\+-polynomial
\vvn.3>
\be
B(x)\>=\,
y_{N-2}(x\+q^{\>2})\,y_{N-3}(x)\,u_{N-1}(x\+q^{-2})\,T_{N-2}(x)+
y_{N-2}(x\+q^{-2})\,y_{N-3}(x\+q^{\>2})\,u_{N-1}(x)\,T_{N-2}(x\+q^{\>2})
\vv.3>
\ee
of type \>$\la_{N-2}+3\+\la_{N-1}+2\+\la_N$.
\end{lem}
\begin{proof}
By Lemma \ref{div}, \>$y_{N-2}(x)$ divides the quasi\+-polynomial $A_{N-2}(x)$,
see \Ref{Ai}, of type \>$\la_{N-2}+2\+\la_{N-1}+3\+\la_N$.
Relations \Ref{abi} and \Ref{W-1} yield
\vvn.2>
\begin{align*}
B(x)\,=\,A_{N-2}(x)\>
\bigl(a_{N-1}(x)\,u_{N-1}(x)+b_{N-1}(x)\,u_{N-1}(x\+q^{-2})\bigr) &{}+{}
\\[3pt]
{}+\,y_{N-2}(x)\,T_{N-1}(x)\>\bigl(\+b_{N-1}(x)-a_{N-1}(x)\bigr) &{}\,,
\\[-24pt]
\end{align*}
which proves Lemma \ref{div2}.
\end{proof}

Assume that the quasi\+-polynomials \>$u_{i+1}\lc u_N$ are constructed
already, and the following properties hold.
\vsk.3>
A.\enspace
For any $j=i\lc N\?$,
\vvn.3>
\beq
\label{Wj}
W_{N\<-j\>}[\>u_{j+1}\lc u_N\+](x)\,=\,y_j(x)\,Q_{N\<-j}(x)\,.
\eeq

\vsk.4>
B.\enspace
For any $j=i+1\lc N\?$, there is a quasi\+-polynomial \>$w_{i+1,\+j}(x)$
of type $\la_{i+1}\<\lsym+\la_N-\la_j$ such that
\vvn.4>
\beq
\label{wijQ}
W_{N\<-\+i-1\>}[\>u_{i+1}\lc u_{j-1},u_{j+1}\lc u_N\+](x)\,=\,
w_{i+1,\+j}(x)\,Q_{N\<-\+i-1}(x)\,.
\vvn.4>
\eeq
In particular, \>$w_{i+1,\+i+1}(x)=y_{i+1}(x)$, see \Ref{Wj} for $j=i+1$.

\vsk.5>
C.\enspace
For any $j=i+1\lc N\?$, the quasi\+-polynomial \>$y_i(x)$ divides
the quasi\+-polynomial
\vvn.4>
\be
B_{ij}(x)\,=\,
y_i(x\+q^{\>2})\,y_{i-1}(x)\,w_{i+1,\+j}(x\+q^{-2})\,T_i(x)+
y_i(x\+q^{-2})\,y_{i-1}(x\+q^{\>2})\,w_{i+1,\+j}(x)\,T_i(x\+q^{\>2})
\kern-.6em
\vv.4>
\ee
of type \>$\la_i-\la_j+3\>(\la_{i+1}\<\lsym+\la_N)$.
In particular, \>$B_{i+1,\+i+1}(x)=A_{i+1}(x)$, see \Ref{Ai}.

\vsk.5>
For \>$i=N-2$, property A coincide with formula \Ref{W-1}, property B
is straightforward: \>$w_{N-1,N-1}=u_N\?$, \>$w_{N-1,N}=u_{N-1}$, and property C
follows from Lemma \ref{div} for \>$i=N-2$, and Lemma \ref{div2}.

\vsk.2>
Define the quasi\+-polynomials \>$w_{ij}(x)$, \,$j=i\lc N\?$, by the rule
\vvn.3>
\beq
\label{wij}
w_{ij}\,=\,\Fe[\>y_i,y_{i-1}\>\wc_{i+1,\+j}\>T_i\+]\,,
\vv.3>
\eeq
see \Ref{Fx}, where \,$\wc_{i+1,\+j}(x)=\wc_{i+1,\+j}(x\+q^{-2})$.
Property C and Proposition \ref{Vy} yield
\vvn.3>
\beq
\label{Ww}
W_2[\>w_{ij},y_i\+](x)\,=\,y_{i-1}(x)\,w_{i+1,\+j}(x\+q^{-2})\,T_i(x)\,.
\vv-.2>
\eeq

\begin{lem}
\label{div3}
The quasi\+-polynomial \,$y_i(x)\<$ divides the quasi\+-polynomial
\vvn-.5>
\rlap{$\sum_{j=i+1}^N\?(-1)^{\>j}\>w_{ij}(x)\>u_j(x)$}\\
of type \>$\la_i\<\lsym+\la_N$.
\end{lem}
\begin{proof}
Let \>$b_i=\Be[\+y_i\+]$, see \Ref{ABy}. By \Ref{wij} and \Ref{Fx},
\vvn.1>
\begin{align*}
\sum_{j=i+1}^N(-1)^{\>j}\,w_{ij}(x)\>u_j(x)\,=\,y_i &{}(x)\>
\sum_{j=i+1}^N(-1)^{\>j}\,\Je[\>y_i,y_{i-1}\>\wc_{i+1,\+j}\>T_i\+](x)\>+{}
\\[3pt]
{}\++\,b_i &{}(x)\>
\sum_{j=i+1}^N(-1)^{\>j}\>y_{i-1}(x)\,T_i(x)\,w_{i+1,\+j}(x\+q^{-2})\,u_j(x)\,.
\\[-15pt]
\end{align*}
Then formula \Ref{wijQ}, Lemma \ref{Wid2} and formula \Ref{Wj} for \>$j=i$ give
\vvn.3>
\beq
\label{wu-}
\sum_{j=i+1}^N(-1)^{\>j-\+i-1}\>w_{i+1,\+j}(x\+q^{-2})\>u_j(x)\,=\,
\frac{W_{N\<-\+i\>}[\>u_{i+1}\lc u_N\+](x)}{Q_{N\<-\+i-1}(x\+q^{-2})}
\,=\,y_i(x)\<\prod_{k=i+1}^{N-1}\!T_k(x)\,,
\vv.3>
\eeq
which proves the lemma.
\end{proof}

\vsk.1>
Set \>$c_{ij}(x)=w_{ij}(x)/y_i(x)$, \,$j=i\lc N\?$, and
\vvn.2>
\beq
\label{ui}
u_i(x)\,=\,\sum_{j=i+1}^N(-1)^{\>j-\+i-1}\,c_{ij}(x)\>u_j(x)\,.
\vv.2>
\eeq
By Lemma \ref{div3}, \,$u_i(x)$ is a quasi\+-polynomial of type \>$\la_i$.

\begin{prop}
\label{A}
$\;W_{N\<-\+i+1\>}[\>u_i\lc u_N\+](x)\,=\,y_{i-1}(x)\,Q_{N\<-\+i+1}(x)$.
\end{prop}
\begin{proof}
By \Ref{Ww} and \Ref{wijQ},
\begin{align*}
c_{ij}(x)-c_{ij}(x\+q^{-2})\, &{}=\,
\frac{y_{i-1}(x)\,w_{i+1,\+j}(x\+q^{-2})\,T_i(x)}{y_i(x)\,y_i(x\+q^{-2})}
\\[8pt]
& {}=\,\frac{y_{i-1}(x)\,T_i(x)\,
W_{N\<-\+i-1\>}[\>u_{i+1}\lc u_{j-1},u_{j+1}\lc u_N\+](x\+q^{-2})}
{y_i(x)\,y_i(x\+q^{-2})\,Q_{N\<-\+i-1}(x\+q^{-2})}\;.
\\[-14pt]
\end{align*}
So Lemma \ref{Wid2} yields
\vvn.3>
\be
\sum_{j=i+1}^N(-1)^{\>j}\,\bigl(c_{ij}(x)-c_{ij}(x\+q^{-2})\bigr)
\>u_j(x\+q^{-2\+l})\,=\,0\,,\qquad l=1\lc N\<-i-1\,,
\vv-.3>
\ee
and
\vvn-.3>
\begin{align*}
\sum_{j=i+1}^N(-1)^{N\<-j}\,\bigl(c_{ij}(x)-c_{ij}(x\+q^{-2})\bigr)
\>u_j(x\+q^{\>2(i-N)})\,={}&
\\[8pt]
{}=\,\frac{y_{i-1}(x)\,T_i(x)\,W_{N\<-\+i\>}[\>u_{i+1}\lc u_N\+](x\+q^{-2})}
{y_i(x)\,y_i(x\+q^{-2})\,Q_{N\<-\+i-1}(x\+q^{-2})}\,={}&
\,\frac{y_{i-1}(x)\,Q_{N\<-\+i+1}(x)}{W_{N\<-\+i\>}[\>u_{i+1}\lc u_N\+](x)}\;,
\\[-10pt]
\end{align*}
where the last equality also uses formula \Ref{Wj} for \>$j=i$ and
formula \Ref{Qk}. These relations together with \Ref{ui} give
\vvn-.1>
\beq
\label{uc}
u_i(x\+q^{-2\+l})\,=\,
\sum_{j=i+1}^N(-1)^{\>j-\+i-1}\,c_{ij}(x)\>u_j(x\+q^{-2\+l})\,,
\qquad l=1\lc N\<-i-1\,,\kern-2em
\vv-.4>
\eeq
and
\vvn-.2>
\beq
\label{ucy}
u_i(x\+q^{\>2(i-N)})\,=\,\frac{(-1)^{N\<-i}\>y_{i-1}(x)\,Q_{N\<-\+i+1}(x)}
{W_{N\<-\+i\>}[\>u_{i+1}\lc u_N\+](x)}\>+
\sum_{j=i+1}^N(-1)^{\>j-\+i-1}\,c_{ij}(x)\>u_j(x\+q^{\>2(i-N)})\,.
\vv.3>
\eeq
Using equalities \Ref{uc}, \Ref{ucy} in the definition of
\>$W_{N\<-\+i+1\>}[\>u_i\lc u_N\+]$, see \Ref{Wr} completes the proof
of the proposition.
\end{proof}

\begin{prop}
\label{BC}
For any \>$j=i\lc N\?$,
\vvn.3>
\be
W_{N\<-\+i\>}[\>u_i\lc u_{j-1},u_{j+1}\lc u_N\+](x)\,=\,
w_{ij}(x)\,Q_{N\<-\+i}(x)\,,
\vvn.4>
\ee
and the quasi\+-polynomial \,$y_{i-1}(x)\<$ divides the quasi\+-polynomial
\vvn.4>
\be
B_{i-1,\+j}(x)\,=\,
y_{i-1}(x\+q^{\>2})\,y_{i-2}(x)\,w_{ij}(x\+q^{-2})\,T_{i-1}(x)+
y_{i-1}(x\+q^{-2})\,y_{i-2}(x\+q^{\>2})\,w_{ij}(x)\,T_{i-1}(x\+q^{\>2})
\kern-.6em
\vv.4>
\ee
of type \>$\la_{i-1}-\la_j+3\>(\la_i\<\lsym+\la_N)$.
\end{prop}
\begin{proof}
Using \Ref{uc} in the definition of
\>$W_{N\<-\+i\>}[\>u_i\lc u_{j-1},u_{j+1}\lc u_N\+](x)$, we get
\vvn.4>
\be
W_{N\<-\+i\>}[\>u_i\lc u_{j-1},u_{j+1}\lc u_N\+](x)\,=\,
c_{ij}(x)\,W_{N\<-\+i\>}[\>u_{i+1}\lc u_N\+](x)\,=\,
w_{ij}(x)\,Q_{N\<-\+i}(x)\,.
\ee

\vsk.3>
By Lemma \ref{div}, \>$y_{i-1}(x)$ divides the quasi\+-polynomial $A_{i-1}(x)$,
\vvn.08>
see \Ref{Ai}, of type \>\>$\la_{i-1}+2\+\la_i+3\>(\la_{i+1}\<\lsym+\la_N)$.
Relations \Ref{abi} and \Ref{Ww} yield
\vvn.4>
\begin{align*}
B_{i-1,\+j}(x)\,={} &\,
A_{i-1}(x)\>\bigl(a_i(x)\,w_{ij}(x)+b_i(x)\,w_{ij}(x\+q^{-2})\bigr)+{}
\\[4pt]
&{}\?+\,y_{i-1}(x)\,T_i(x)\,w_{i+1,\+j}(x\+q^{-2})\>
\bigl(\+b_i(x)-a_i(x)\bigr)\,,
\\[-13pt]
\end{align*}
which proves the second part of the proposition.
\end{proof}

Propositions \ref{A} and \ref{BC} shows that properties A\>--\,C with \>$i$
replaced by \>$i-1$ are valid. So we can construct recursively all
quasi\+-polynomials \>$u_1\lc u_N$ satisfying relations \Ref{Wuu}.

\subsection{Proof of Theorem \ref{main}}
\label{s-Te}
In this section we will show that the sequence of monic polynomials
\>$\Te=(T_1\lc T_{N-1},1)$ is a preframe of the collection of quasi\+-polynomials
\>$\Ue=(u_1\lc u_N)$ constructed in Section \ref{s-Ue}.

\vsk.2>
It is shown in Section \ref{s-Ue} that the quasi\+-polynomials \>$u_1\lc u_N\<$
satisfy relations \Ref{Wuu}, see Proposition \ref{A}. Moreover, for any
\>$1\le i\le j\le N\?$, there is a quasi\+-polynomial \>$w_{ij}(x)$ of type
\>$\la_i\lsym+\la_N-\la_j$ such that
\vvn.4>
\beq
\label{WQ}
W_{N\<-\+i\>}[\>u_i\lc u_{j-1},u_{j+1}\lc u_N\+](x)\,=\,
w_{ij}(x)\,Q_{N\<-\+i}(x)\,,
\vvn.4>
\eeq
see Proposition \ref{BC}, and
\vvn.3>
\beq
\label{W2}
W_2[\>w_{ij},y_i\+](x)\,=\,y_{i-1}(x)\,w_{i+1,\+j}(x\+q^{-2})\,T_i(x)\,,
\vv.2>
\eeq
cf.~\Ref{Ww}.

\begin{lem}
\label{www}
For any \,$1\le i\le j<k\le N\?$, there is a quasi\+-polynomial
\,$\wt_{i\+;\+j,\+k}(x)$ of type \>$\la_i\<\lsym+\la_N-\la_j\<-\la_k$
such that
\vvn.3>
\be
W_2[\>w_{ij},w_{ik}\+](x)\,=\,y_{i-1}(x)\,\wt_{i\+;\+j,\+k}(x)\,T_i(x)\,.
\vv.3>
\ee
\end{lem}
\begin{proof}
For \>$j=i$, we have \>$w_{ii}(x)=y_i(x)$ and
\>$\wt_{i\+;\+i,\+k}(x)=-\>w_{i+1,\+k}(x\+q^{-2})$ by \Ref{W2}, so
the statement holds. For $j>i$, by Lemma \ref{Wid2} and formula \Ref{abi}
we have
\vvn.3>
\begin{align*}
W_2[\>w_{ij},w_{ik}\+](x)\,=\,{}& \bigl(a_i(x)\,w_{ik}(x)+
b_i(x)\,w_{ik}(x\+q^{-2})\bigr)\,W_2[\>w_{ij},y_i\+](x)-{}
\\[4pt]
{}-{}\,& \bigl(a_i(x)\,w_{ij}(x)+b_i(x)\,w_{ij}(x\+q^{-2})\bigr)\,
W_2[\>w_{ik},y_i\+](x)\,,
\end{align*}
which proves Lemma \ref{www}.
\end{proof}

\begin{prop}
\label{Wwt}
For any \,$1\le i\le j_1\?\lsym<j_k\le N\?$, there is a quasi\+-polynomial
\rlap{\,$\wt_{i\+;\+j_1\<\lc\+j_k}(x)$}\\ of type
\>$\la_i\<\lsym+\la_N-\la_{j_1}\!\<\lsym-\la_{j_k}\?$ such that
\be
W_k[\>w_{i,\+j_1}\lc w_{i,\+j_k}\+](x)\,=\,
\wt_{i\+;\+j_1\<\lc\+j_k}(x)\,
\prod_{l=0}^{k-2}\,\Bigl(\+y_{i-1}(x\+q^{-2\+l})\>
\prod_{m=i}^{i+\+l}\>T_m(x\+q^{-2\+l})\Bigr)\,.
\vv.2>
\ee
\end{prop}
\begin{proof}
We prove the statement by induction with respect to \>$k$.
The case \>$k=2$ \>is the base of induction, see Lemma \ref{www}.

\vsk.2>
Set \>$f_{ij}=W[\>y_i,w_{ij}\+]$. Assume that \>$j_1\<=i$.
By Lemma \ref{Wid} and formula \Ref{W2},
\vvn.6>
\begin{align*}
W_k &{}[\>y_i,w_{i,\+j_2}\lc w_{i,\+j_k}\+](x)\,=\,
\frac{W_{k-1}[\>f_{i,\+j_2}\lc f_{i,\+j_k}\+](x)}
{y_i(x\+q^{-2})\ldots y_i(x\+q^{-2\+k})}\,={}
\\[8pt]
& {}=\,(-1)^{k-1}\>
W_{k-1}[\>w_{i+1,\+j_2}\lc w_{i+1,\+j_k}\+](x\+q^{-2})\,\prod_{l=0}^{k-2}\>\>
\frac{y_{i-1}(x\+q^{-2\+l})\,T_i(x\+q^{-2\+l})}{y_i(x\+q^{-2\+(l\++1)})}
\\[8pt]
& {}=\,(-1)^{k-1}\>\wt_{i+1\+;\+j_2\<\lc\+j_k}(x)\,
\prod_{l=0}^{k-2}\,\Bigl(\+y_{i-1}(x\+q^{-2\+l})\>
\prod_{m=i}^{i+\+l}\>T_m(x\+q^{-2\+l})\Bigr)\,,
\\[-10pt]
\end{align*}
where for the last equality we use the induction assumption.
For $j_1\<>i$, by Lemma \ref{Wid2} and formula \Ref{abi} we have
\vvn.4>
\begin{align*}
& W_k[\>w_{i,\+j_1}\lc w_{i,\+j_k}\+](x)\,={}
\\[6pt]
&\!\<{}=\,\sum_{l=1}^k\,(-1)^{\>l\+-1}
\bigl(a_i(x)\,w_{i,\+j_l}(x)+b_i(x)\,w_{i,\+j_l}(x\+q^{-2})\bigr)\,
W_k[\>y_i,w_{i,\+j_1}\lc w_{i,\+j_{l-1}},
w_{i,\+j_{l+1}}\lc w_{i,\+j_k}\+](x)\,,
\\[-12pt]
\end{align*}
which proves the proposition.
\end{proof}

To complete the proof of Theorem \ref{main} we should show that
for any $k=1\lc N-1$, and any \>\$k\+$-\+ele\-ment subset
$\{\+i_1\lc i_k\}\subset\{1\lc N\+\}$, the polynomial \>$Q_k(x)$
divides the quasi\+-polynomial $W_k[\>u_{i_1}\lc u_{i_k}](x)$ of type
$\la_{i_1}\!\lsym+\+\la_{i_k}$.

\vsk.2>
Let $\{\+j_1\lc j_{N\<-k}\}$ be the complement
of $\{\+i_1\lc i_k\}$ in $\{1\lc N\+\}$. Then by formula \Ref{WQ},
Lemmas \ref{common}, \ref{Wid3}, and formula \Ref{Wuu},
\vvn.5>
\begin{align*}
W_k[\>u_{i_1}\lc u_{i_k}](x)\,&{}=\,\const\cdot
x^{\>(k\++1-N)(\la_1\<\lsym+\>\la_N)}\>\x{}
\\[-6pt]
&\>{}\x\,W_{N\<-\+k}[w_{\+1,\+j_1}\lc w_{\+1,\+j_{N\<-\+k}}]
(x\+q^{\>2\+(N\<-\+k\+-1)})\;Q_{N\<-1}(x)\,\prod_{l=1}^{N\<-\+k\+-1}\>
\frac{Q_{N-1}(x\+q^{\>2\+l})}{Q_N(x\+q^{\>2\+l})}\;.
\\[-22pt]
\end{align*}
Then by Proposition \ref{Wwt},
\vvn.2>
\be
W_k[\>u_{i_1}\lc u_{i_k}](x)\,=\,\const\cdot
\wt_{\+1;\+j_1\<\lc\+j_{N\<-\+k}}(x\+q^{\>2\+(N\<-\+k\+-1)})\,Q_k(x)\,,
\vv.3>
\ee
where the quasi\+-polynomial \,$\wt_{\+1;\+j_1\<\lc\+j_{N\<-\+k}}$ has type
$\la_{i_1}\!\lsym+\+\la_{i_k}$. Theorem \ref{main} is proved.

\section{Difference operators}
\label{s-diff}
\subsection{Difference operator of a collection of quasi\+-polynomials}
\label{s-dipol}
Recall that \>$q\ne0,\pm\+1$. Denote by \>$\tau$ the multiplicative shift
operator that acts on functions of \>$x$ by the rule
\vvn.2>
\be
(\tau f)(x)\,=\,f(x\+q^{-2})\,.
\vv.2>
\ee
A function \>$f(x)$ is called a quasi\+-constant if \>$\tau f=f$.

\vsk.2>
An operator $D=a_0(x)+a_1(x)\,\tau\lsym+a_N(x)\,\tau^{\+N}$, where
\>$a_0\lc a_N$ are functions and $a_N$ is not identically zero, is called
a {\it difference operator of order\/} $N\?$. The functions \>$a_0\lc a_N$ are
the {\it coefficients\/} of \>$D$. If \>$a_N\<=1$, the operator \>$D$ is called
{\it monic\/}.

\vsk.2>
Recall that for any collection of quasi\+-polynomials \>$\Ue=(u_1\lc u_N)$
by definition in Section \ref{s-pol}, \>$W_{N\>}[\>u_1\lc u_N\+]\ne0$.

\begin{lem}
\label{DU}
For any collection of quasi\+-polynomials \,$\Ue=(u_1\lc u_N)$ there exists
a unique monic difference operator \,$\DU$ such that \,$\DU\,u_i=0$ for all
\>$i=1\lc N$.
\end{lem}
\begin{proof}
Define the operator \>$\DU$ by the rule
\vvn.2>
\beq
\label{DUf}
\DU\,f\,=\,\frac{W_{N+1\>}[\>u_1\lc u_N,f\>]}{W_{N\>}[\>u_1\lc u_N\+]}\;.
\vv.2>
\eeq
Then clearly \>$\DU\,u_i=0$ for all \>$i=1\lc N$. On the other hand, write
\>$\DU\<=a_0\lsym+a_{N-1}\,\tau^{\+N-1}\?+\tau^{\+N}\<$. Then equalities
\>$\DU\,u_i=0$, \,$i=1\lc N\?$, amount to a system of linear equations on
\>$a_0\lc a_{N-1}$:
\vvn.1>
\beq
\label{aj}
a_0\,u_i\<+a_1\>\tau u_i\<\lsym+a_{N-1}\,\tau^{\+N-1}u_i
\>=\,-\>\tau^{\+N}u_i\,,\qquad i=1\lc N\>.\kern-2em
\vv.2>
\eeq
Since \>$W_{N\>}[\>u_1\lc u_N\+]\ne0$, the matrix
\>$(\tau^{j-1}\>u_i)_{i,j=1}^N$ \>is invertible
and solution of system \Ref{aj} has is unique.
\end{proof}

\vsk-.2>
The operator \>$\DU$ is called the {\it fundamental difference operator\/}
of the collection \>$\Ue$. Notice that by \Ref{DUf},
\vvn-.6>
\beq
\label{a0}
a_0(x)\>=\,(-1)^N\;
\frac{W_{N\>}[\>u_1\lc u_N\+](x\+q^{-2})}{W_{N\>}[\>u_1\lc u_N\+](x)}\;.
\vv.2>
\eeq

\begin{lem}
\label{Duf}
Let \,$\Ue=(u_1\lc u_N)$ be a collection of quasi\+-polynomials. Any solution
\>$f$ of the difference equation \>$\DU\>f=0$ is a linear combination of
\,$u_1\lc u_N$ with quasi\+-constant coefficients.
\end{lem}
\begin{proof}
Let
\vvn-.8>
\beq
\label{ci}
c_i(x)\,=\,-\>\frac{W_{N\>}[\>u_1\lc u_{i-1},f,u_{i+1}\lc u_N\+](x)}
{W_{N\>}[\>u_1\lc u_N\+](x)}\;,\qquad i=1\lc N\>.\kern-2em
\vv-.2>
\eeq
By Lemma \ref{Wid3}, \,$f\<=c_1\+u_1\lsym+c_N\>u_N$. If \>$c_i\<\ne 0$,
consider a collection \,$\Ute_i\<=(u_1\lc u_{i-1},f,u_{i+1}\lc u_N)$.
Lemma \ref{DU} implies that \>$\DUti_i\!=\DU$, and by formula \Ref{a0},
\vvn.2>
\be
\frac{W_{N\>}[\>u_1\lc u_{i-1},f,u_{i+1}\lc u_N\+](x\+q^{-2})}
{W_{N\>}[\>u_1\lc u_{i-1},f,u_{i+1}\lc u_N\+](x)}\,=\,
\frac{W_{N\>}[\>u_1\lc u_N\+](x\+q^{-2})}{W_{N\>}[\>u_1\lc u_N\+](x)}\;.
\vv.2>
\ee
Thus \>$c_i(x\+q^{-2})=c_i(x)$ for all \>$i=1\lc N\?$.
\end{proof}

\begin{lem}
\label{DUprod}
For a collection of quasi\+-polynomials \,$\Ue=(u_1\lc u_N)$,
set \,$v_N=u_N\<$ and
\vvn.2>
\beq
\label{vi}
v_i\,=\,\frac{W_{N\<-\+i+1\>}[\>u_i\lc u_N\+]}
{W_{N\<-\+i\>}[\>u_{i+1}\lc u_N\+]}\;,\qquad i=1\lc N-1\,.\kern-2em
\vv-.1>
\eeq
Then
\vvn-.2>
\beq
\label{DUv}
\DU\,=\,\Bigl(\+\tau\>-\>\frac{\tau v_1}{v_1}\,\Bigr)\,\ldots\,
\Bigl(\+\tau\>-\>\frac{\tau v_N}{v_N}\,\Bigr)\,.
\vv.2>
\eeq
\end{lem}
\begin{proof}
By applying repeatedly Lemma \ref{Wid}, we get
\vvn.3>
\be
\Bigl(\+\tau\>-\>\frac{\tau v_i}{v_i}\,\Bigr)\,\ldots\,
\Bigl(\+\tau\>-\>\frac{\tau v_N}{v_N}\,\Bigr)\>f\,=\,
\frac{W_{N\<-\+i+2\>}[\>u_i\lc u_N,f\>]}{W_{N\<-\+i+1\>}[\>u_i\lc u_N\+]}\,,
\qquad i=1\lc N\>.\kern-2em
\vv.2>
\ee
Comparing this formula with formula \Ref{DUf} completes the proof.
\end{proof}

\begin{cor}
\label{DUrat}
Let \,$\Ue$ be a regular collection of quasi\+-polynomials. Then the operator
\,$\DU$ has rational coefficients.
\end{cor}
\begin{proof}
For a regular collection of quasi\+-polynomials \>$\Ue$, the expressions
\>$\tau v_i\+/\<v_i\<$ in formula \Ref{DUv} are rational functions,
which proves the claim.
\end{proof}

\begin{prop}
\label{semi}
Let \,$\Ue=(u_1\lc u_N)$ be a collection of quasi\+-polynomials such that
the operator \,$\DU$ has rational coefficients. Then the collection \,$\Ue\<$
is semiregular, that is, the quasi\+-polynomial \,$W_{N\>}[\>u_1\lc u_N\+]$ is
\>{\rm log}\<-free.
\end{prop}

\nt
We will prove Proposition \ref{semi}, as well as
Propositions \ref{DUreg}\>\>--\,\ref{DUtgen} below, in Section \ref{s-proofs}.

\begin{prop}
\label{DUreg}
Let \,$\Ue$ be a collection of quasi\+-polynomials such that the operator
\,$\DU$ has rational coefficients. Then there is a regular collection
of quasi\+-polynomials \,$\Ute$ such that \,$\DUt=\DU$.
\end{prop}

Say that $\lab=(\la_1\lc\la_N)$ is {\it dominance\+-free\/} \>if
\>$q^{\>2(\la_i-\la_j)}\<\ne q^{\>2s}$ for all \>$1\le i<j\le N$
and \>$s\in\Z_{\ge1}$.

\begin{prop}
\label{DUregla}
Let \,$\Ue$ be a collection of quasi\+-polynomials of type $\lab$ such that
the operator \,$\DU$ has rational coefficients. Assume that $\lab$ is
dominance\+-free, or \>$q$ is a root of unity. Then there is a regular
collection of quasi\+-polynomials \,$\Ute$ of type $\lab$ such that \,$\DUt=\DU$.
\end{prop}

\begin{example}
Let \>$\Ue=(u_1,u_2)$ be the collection of quasi\+-polynomials of type
$\lab=(1,0)$:
\vvn.1>
\be
u_1(x)\>=\>x\,,\qquad
u_2(x)\>=\>\frac1{1-q^{-2}}+\frac{q^2\+x\+\log x}{2\>\log\+q}\;.
\vv-.3>
\ee
Then
\vvn-.4>
\be
W_2[\>u_1,u_2\+](x)\,=\,x-x^2\>,\qquad
\DU\>=\,\tau^2\<-\Bigl(q^{-2}\?+\frac{1-x\+q^{-2}}{1-x}\+\Bigr)\,\tau\++\+
q^{-2}\>\frac{1-x\+q^{-2}}{1-x}\;,
\ee
and \>$\Ute=(u_2,u_1)$ is a regular collection of quasi\+-polynomials such that
\vvn.08>
\>$\DUt=\DU$. However if \>$q$ is not a root of unity, \>$\DU$ is not the
fundamental difference operator of any regular collection of quasi\+-polynomials
of type $\lab=(1,0)$. If \>$q^{\>2\+\ell}\?=1$ for a positive integer $\ell$,
set \>$u'_2(x)=x^{\+\ell}\+u_2(x)$. Then \>$\Ute'\?=(u'_2,u_1)$ is a regular
collection of quasi\+-polynomials of type $\lab=(1,0)$ such that \>$\DUtp\<=\DU$.
\end{example}

Say that $\lab=(\la_1\lc\la_N)$ is {\it generic\/} if
\>$q^{\>2(\la_i-\la_j)}\<\ne q^{\>2s}$ for all \>$1\le i<j\le N$
and \>$s\in\Z$.

\begin{prop}
\label{DUgen}
Let \,$\Ue=(u_1\lc u_N)$ be a collection of quasi\+-polynomials of type $\lab$.
such that the operator \,$\DU$ has rational coefficients. Assume that $\lab$
is generic. Then the quasi\+-polynomials \,$u_1\lc u_N$ are \>{\rm log}\<-free.
In particular, the collection \,$\Ue$ is regular.
\end{prop}

\begin{prop}
\label{DUtgen}
Let \,$\Ue=(u_1\lc u_N)$ and \,$\Ute=(\ut_1\lc\ut_N)$ be collections of
quasi\+-polynomials of type $\lab$ such that \,$\DU=\DUt\<$ and the operator
\,$\DU$ has rational coefficients. Assume that $\lab$ is generic.
Then there are quasi\+-constants \,$c_1\lc c_N\?$ such that \,$\ut_i=c_i\>u_i$,
\,$i=1\lc N\?$. If \,$q$ is not a root of unity, then the quasi\+-constants
\,$c_1\lc c_N\?$ are constants.
\end{prop}

\subsection{Difference operator of a solution of Bethe ansatz equations}
\label{s-diffop}
Fix collections of complex numbers $\lab=(\la_1\lc\la_N)$, nonnegative integers
\,$\lb=(l_1\lc l_{N-1})$ and monic polynomials \>$\Tb=(T_1\lc T_{N-1})$.

\vsk.2>
Given \>$\tb\in\Cl\<$, define the polynomials \>$p_1(x)\lc p_{N-1}(x)$
by the rule
\beq
\label{ptb}
p_i(x)\,=\,\prod_{j=1}^{l_i}\,(x-t_j^{(i)})\,,\qquad i=1\lc N-1\,,\kern-2em
\vv.2>
\eeq
cf.~\Ref{px}, and set \>$p_0(x)=p_N(x)=1$. Let
\vvn.2>
\beq
\label{Rt}
R_i(x)\,=\,\frac{p_{i-1}(x)}{p_i(x)}\;
\prod_{j=i}^{N-1}\,T_j(x\+q^{\>2(i-j)})\,,\qquad i=1\lc N\>,\kern-2em
\vv.2>
\eeq
Define the {\it fundamental difference operator\/} $D^\ts\?$ of the point
\,$\tb$ by the rule
\vvn.3>
\beq
\label{DtR}
D^\ts\>=\,\Bigl(\+\tau\>-\>q^{-2\la_1}\,\frac{R_1(x\+q^{-2})}
{R_1(x)}\,\Bigr)\,\ldots\,\Bigl(\+\tau\>-\>q^{-2\la_N}\,
\frac{R_N(x\+q^{-2})}{R_N(x)}\,\Bigr)\,.
\eeq

\begin{thm}
\label{DUt}
Let \,$\tb\in\Cl\<$ be a solution of equations \>\Ref{Be} associated with
the data \,$\lb\+,\+\Tb\<,\+\lab$. Let \,$\Ue$ \>be a collection of
quasi\+-polynomials of type $\lab$ such that \,$\Te=(T_1\lc T_{N-1},\>1)$ is
a preframe of \;$\Ue$ and the \$\Sl$-orbit of \;$\tb$ equals \>$X_{\Ue,\Te}$.
Then \,$D^\ts\<=\DU$.
\end{thm}
\begin{proof}
Let \>$\Ue=(u_1\lc u_N)$. Since the \$\Sl$-orbit of \;$\tb$ equals
\>$X_{\Ue,\Te}$, formulae \Ref{QT}, \Ref{defy} imply
\vvn-.1>
\be
x^{\+\la_i}R_i(x)\,=\,\frac{W_{N\<-\+i+1\>}[\>u_i\lc u_N\+](x)}
{W_{N\<-\+i\>}[\>u_{i+1}\lc u_N\+](x)}\;,
\vv.5>
\ee
Thus \>$D^\ts\<=\DU$, see formulae \Ref{vi}, \Ref{DUv}, \Ref{DtR}.
\end{proof}

Recall that a point \>$\tb\in\Cl\<$ is admissible if it satisfies conditions
\Ref{cond1}.

\begin{cor}
\label{DtU}
Let \,$\tb\in\Cl\<$ be an admissible solution of equations \>\Ref{Be}
associated with the data \,$\lb\+,\+\Tb\<,\+\lab$. Then there exists
a collection of quasi\+-polynomials \,$\Ue=(u_1\lc u_N)$ of type \>$\lab$
such that \,$D^\ts\<=D_{\+\Ue}$ and
\vvn-.3>
\beq
\label{WNT}
W_{N\>}[\>u_1\lc u_N\+](x)\,=\,x^{\+\la_1\lsym+\>\la_N}\,
\prod_{i=1}^N\,\prod_{j=0}^{N\?-i}\,T_{N\<-\+i\++1}(x\+q^{-2j})\,.
\eeq
\end{cor}
\begin{proof}
\label{degU}
By Theorem \ref{main}, the collection of quasi\+-polynomials \,$\Ue=(u_1\lc u_N)$
constructed in Section \ref{s-Ue} has a preframe \,$\Te=(T_1\lc T_{N-1},\>1)$
and the \$\Sl$-orbit of \;$\tb$ equals \>$X_{\Ue,\Te}$. Hence, formula
\Ref{WNT} holds, cf.~\Ref{WN}, and \>$D^\ts\<=D_{\+\Ue}$ by Theorem \ref{DUt}.
\end{proof}

\begin{thm}
\label{UT}
Let \,$\tb\in\Cl\<$ be an admissible solution of equations \>\Ref{Be}
associated with the data \,$\lb\+,\+\Tb\<,\+\lab$. Let \,$\Ue$ be a collection
of quasi\+-polynomials of type \>$\lab$ such that \,$D^\ts\<=D_{\+\Ue}$.
Assume that \,$q$ is not a root of unity and \>$\lab$ is generic. Then
\,$\Te=(T_1\lc T_{N-1},1)$ is a preframe of \,$\Ue$ and the \$\Sl$-orbit
of \;$\tb$ equals \>$X_{\Ue,\Te}$.
\end{thm}
\begin{proof}
Let \>$\Ue=(u_1\lc u_N)$, and \>$\Ute=(\ut_1\lc\ut_N)$ be the collection
of quasi\+-polynomials of type \>$\lab$ constructed in the proof of
Theorem \ref{main}, see Section \ref{s-Ue}. By Theorem \ref{DUt}
\>$D_{\+\Ute}=D^\ts\<=\DU$. Then by Lemma \ref{DUtgen}, for any \>$i=1\lc N\?$,
the quasi\+-polynomials \>$u_i$ and \>$\ut_i$ are proportional, which proves
the theorem.
\end{proof}

Theorem \ref{UT} shows that for \>$q$ being not a root of unity, generic
\>$\lab$, and an admissible solution \>$\tb$ of the Bethe ansatz equations
\Ref{Be} solving the difference equation \>$D^\ts\<f=0$ allows one to produce
a collection of quasi\+-polynomials \>$\Ue$ of type \>$\lab$ such that the
\$\Sl$-orbit of \;$\tb$ equals \>$X_{\Ue,\Te}$. This is an alternative way
to the construction of such a collection of quasi\+-polynomials given in
Section \ref{s-Ue}. Moreover, Theorem \ref{UT} and Lemma \ref{DUtgen} imply
that such a collection of quasi\+-polynomials is unique up to rescaling of
individual quasi\+-polynomials.

\subsection{Proofs of Propositions \ref{semi}\>\>--\,\ref{DUtgen}}
\label{s-proofs}
Recall that \>$q\ne0,\pm\+1$. Most of the technicalities in this section are
related to the fact that \>$q$ can be a root of unity.

\begin{lem}
\label{qua}
Let \,$c(x)$ be a quasi\+-constant of the form $c(x)=x^{\>\al}\>r(x,\log\+x)$,
where \>$\al\in\C$ and \,$r(x,y)$ is a rational function in \,$x\<$ and
a polynomial in \,$y$. Then \,$r(x,y)$ does not depend on \,$y$.
\end{lem}
\begin{proof}
Let \>$r(x,y)=r_0(x)+r_1(x)\,y\lsym+r_k(x)\,y^k\<$ and \>$k\ge 1$.
The equality \>$c(x)=c(x\+q^{-2})$ imply that
\vvn-.1>
\beq
\label{rk}
r_k(x)=q^{-2\+\al}\>r_k(x\+q^{-2})\,,\qquad
r_{k-1}(x)-q^{-2\+\al}\>r_{k-1}(x\+q^{-2})\>=\>-2\>k\,r_k(x)\,\log\+q\,.
\vv.2>
\eeq
Denote by \>$d_i$ the order of \>$r_i(x)$ at $x=0$, that is,
\,$\lim\nolimits_{\>x\to0}\bigl(x^{-d_i}\>r_i(x)\bigr)\ne 0\+,\8$.
Then the first equality in \Ref{rk} gives \,$q^{\>2\+(\al+d_k)}\<=1$, which
makes impossible matching the orders of the left and right sides in the second
equality.
\end{proof}

For a quasi\+-polynomial \>$f(x)=x^{\>\al}\+\bigl(P_0(x)+P_1(x)\,\log\+x\lsym+
P_k(x)\>(\+\log\+x)^k\bigr)$, where \>$P_0\lc P_k$ are polynomials,
\>$P_k\ne0$, denote by \>$\bra\+f\+\ket(x)=x^{\>\al}\+P_k(x)$ the top part
of \>$f(x)$.

\begin{lem}
\label{top}
Let \,$D$ be a difference operator with rational coefficients.
If a quasi\+-polyno\-mial \>$f(x)$ satisfies the equation \>$Df=0$,
then \>$D\+\bra\+f\+\ket=0$ too.
\end{lem}
\begin{proof}
Since \>$\tau\+\bigl(x^{\>\al}\+x^{\+i}(\+\log\+x)^{\+j}\bigr)=
q^{-2\+(\al+i\+)}\>x^{\>\al}\+x^{\+i}\>r(\+\log\+x)$ for some monic polynomial
\>$r(s)$ of degree \>$j$, the claim follows.
\end{proof}

\begin{proof}[Proof of Proposition \ref{semi}]
Let
\>$c(x)=W_{N\>}[\>u_1\lc u_N\+](x)\big/\Bra\+W_{N\>}[\>u_1\lc u_N\+]\+\Ket(x)$.
By Lemma \ref{top}, \>$c(x)$ is a quasi\+-constant, and \>$c(x)$ satisfies
the assumption of Lemma \ref{qua}. Since the quasi\+-polynomial
\>$\Bra\+W_{N\>}[\>u_1\lc u_N\+]\+\Ket$ is \+log\+-free, Proposition \ref{semi}
follows.
\end{proof}

\begin{proof}[Proof of Propositions \ref{DUreg}, \ref{DUregla}]
We will prove the Propositions by induction on the number of quasi\+-polynomials
in the collection \>$\Ue=(u_1\lc u_N)$.

\vsk.2>
If the quasi\+-polynomial \>$u_N\<$ is not \+log\+-free, let \>$f=\bra\+u_N\ket$.
By Lemma \ref{top}, \>$\DU\>f=0$. If \>$W_{N\>}[\>u_1\lc u_{N-1},f\>]\ne0$, set
\vvn.08>
\>$\Uhe=(u_1\lc u_{N-1},f\+)$. Then the collection \>$\Uhe$ has type $\lab$,
\>$\DUh=\DU$, and the quasi\+-polynomial \>$f$ is \+log\+-free.

\vsk.2>
If \>$W_{N\>}[\>u_1\lc u_{N-1},f\>]=0$, then
\>$W_{N\>}[\>u_1\lc u_{i-1},f,u_{i+1}\lc u_N\+]\ne 0$ \>for some
\>$i=1\lc N-1$, see Lemma \ref{ci}. Then for the collection
\vvn.06>
\>$\Uhe=(u_1\lc u_{i-1},u_N,u_{i+1}\lc u_{N-1},f\+)$, \>$\DUh=\DU$,
and the quasi\+-polynomial \>$f$ is \+log\+-free.

\vsk.2>
The quasi\+-constants \>$c_i(x)$ given by \Ref{ci} have the form
\>$c_i(x)=x^{\+\la_N-\la_i}\+r_i(x)$ for some rational functions \>$r_i(x)$.
Let \>$d_i$ be the order of \>$r_i(x)$ at \>$x=0$. Since \>$c_i(x)$ is
a quasi\+-constant, \>$q^{\>2\+(\la_N-\la_i+\+d_i)}\?=1$. For a dominance\+-free
$\lab$, we have \>$d_i\le 0$, and we set \>$\ell=0$. If \>$q$ is a root
of unity, we find an integer \>$\ell\ge d_i$ such that \>$q^{\>2\+\ell}\?=1$.
Then \>$x^{\+\la_i-\la_N+\+\ell\+-\+d_i}\<$ is a quasi\+-constant, and
\>$\uh_i(x)=x^{\+\la_i-\la_N}\+x^{\+\ell\+-\+d_i}\+u_N(x)$ is
a quasi\+-polynomials of type $\la_i$. Thus
\>$\Uhe'\?=(u_1\lc u_{i-1},\uh_i,u_{i+1}\lc u_{N-1},f\+)$ is a collection of
quasi\+-polynomials of type $\lab$, \>$\DUhp\<=\DU$, and the quasi\+-polynomial
\>$f$ is \+log\+-free.

\vsk.2>
Assume now that the quasi\+-polynomial \>$u_N\<$ is \+log\+-free,
so \>$r(x)=u_N(x\+q^{-2})/u_N(x)$ is a rational function.
Let the functions \>$v_1\lc v_N$ be given by \Ref{vi}. Set
\vvn.4>
\be
D\,=\,\Bigl(\+\tau\>-\>\frac{\tau v_1}{v_1}\,\Bigr)\,\ldots\,
\Bigl(\+\tau\>-\>\frac{\tau v_{N-1}}{v_{N-1}}\,\Bigr)\,=\,
\tau^{\+N-1}\?+b_{N-2}\,\tau^{\+N-2}\?\lsym+b_0\,.
\vv.2>
\ee
By Lemma \ref{DUprod}, the operator \>$\DU$ factors,
\vvn.2>
\be
D\cdot\<\bigl(\+\tau-r(x)\bigr)\,=\,
\DU\,=\,\tau^{\+N}\!+a_{N-1}\,\tau^{\+N-1}\?\lsym+a_0\,,
\vv.1>
\ee
so the coefficients \>$b_0,\lc b_{N-2}$ are determined by the equations
\vvn.3>
\be
a_i(x)\,=\,b_{i-1}(x)-b_i(x)\,r(x\+q^{-2\+i})\,,\qquad i=0\lc N-1\,,
\kern-2em
\vv.2>
\ee
where \>$b_{-1}(x)=0$, and \>$b_{N-1}(x)=1$. Therefore,
\>$b_0(x),\lc b_{N-1}(x)$ are rational functions.

\vsk.2>
Let \>$u'_i=W_2[\>u_i,u_N\+]$, and \>$\Ue'\?=(u'_1\lc u'_{N-1})$.
\vvn.1>
By Lemmas \ref{DUprod} and \ref{DU}, the operator \>$\DUp$ acts as follows
\be
\DUp\>f(x)\,=\,\frac1{u_N(x\+q^{\>2\+(1-N)})}\;D\>\bigl(f(x)\+u_N(x)\bigr)\,,
\vv.3>
\ee
and hence, has rational coefficients. By the induction assumption, there is
a regular collection of quasi\+-polynomials \>$\Ue''\?=(u''_1\lc u''_{N-1})$ such
that \>$\DUpp\<=\DUp$. Set
\vvn.2>
\beq
\label{cij}
c_{ij}(x)\,=\,-\>
\frac{W_{N-1\>}[\>u'_1\lc u'_{j-1},u''_i,u'_{j+1}\lc u'_{N-1}\+](x)}
{W_{N-1\>}[\>u'_1\lc u'_{N-1}\+](x)}\;,\qquad i,j=1\lc N-1\,.\kern-2em
\vv.2>
\eeq
Similarly to the proofs of Proposition \ref{semi} and Lemma \ref{Duf},
one can show that
\vvn.3>
\beq
\label{u'}
u''_i(x)\,=\,c_{i,1}(x)\,u'_1(x)\lsym+c_{i,N-1}(x)\,u'_{N-1}(x)\,,\qquad
i=1\lc N-1\,,\kern-2em
\vv.3>
\eeq
the Wronskians in \Ref{cij} are \+log\+-free quasi\+-polynomials,
and \>$c_{ij}(x)$ are quasi\+-constants of the form
\>$c_{ij}(x)=x^{\>\al_{ij}}\+r_{ij}(x)$ for some rational functions
\>$r_{ij}(x)$. Let \>$d_{ij}$ be the order of \>$r_{ij}(x)$ at \>$x=0$.
Since \>$c_{ij}(x)$ is a quasi\+-constant, \>$q^{\>2\+(\al_{ij}+\+d_{ij})}\?=1$,
and the function \>$x^{-d_{ij}}\+r_{ij}(x)$ is also a quasi\+-constant.
Denote by \>$P(x)$ the least common denominator of the functions
\>$x^{-d_{ij}}\+r_{ij}(x)$, \>$i,j=1\lc N-1$. The polynomial \>$P(x)$
is a quasi\+-constant as well.

\vsk.2>
Define the collection of quasi\+-polynomials \>$\Ute=(\ut_1\lc\ut_N)$ by
the rule: \>$\ut_N=u_N$ and
\vvn.3>
\be
\ut_i(x)\,=\,P(x)\bigl(c_{i,1}(x)\,u_1(x)\lsym+c_{i,N-1}(x)\,u_{N-1}(x)\bigr)
\,,\qquad i=1\lc N-1\,.\kern-2em
\vv.3>
\ee
Then \>$W_2[\>\ut_i,\ut_N](x)=P(x)\,u''_i(x)$, \>$i=1\lc N-1$, see \Ref{u'}.
By Lemmas \ref{common}, \ref{Wid},
\vvn.3>
\be
W_{N-\+i\++1\>}[\>\ut_i\lc\ut_N\+](x)
\prod_{j=1}^{N-\+i\+-1}\?\ut_N(x\+q^{-2\+j})\,=\,
W_{N-\+i\>}[\>u''_i\lc u''_{N-1}\+](x)\;\prod_{k=0}^{N-i}\>P(x\+q^{-2\+k})\,,
\vv.2>
\ee
$i=1\lc N-2$, which implies that the quasi\+-polynomials
\>$W_{N-\+i\++1\>}[\>\ut_i\lc\ut_N\+]$ are \+log\+-free,
and \>$\Ute$ is a regular collection of quasi\+-polynomials.
Since \>$\DUt=\DU$ by Lemma \ref{DU}, Proposition \ref{DUreg} is proved.
\end{proof}

\begin{proof}[Proof of Proposition \ref{DUtgen}]
By Proposition \ref{semi}, the quasi\+-polynomial \>$W_{N\>}[\>u_1\lc u_N\+]$
is \+log\+-free. Since \,$\DU\>\ut_i\<=0$, the quasi\+-polynomials
\>$W_{N\>}[\>u_1\lc u_{j-1},\ut_i,u_{j+1}\lc u_N\+]$ are \+log\+-free as well.
Set
\vvn-.3>
\be
c_{ij}(x)\,=\,-\>
\frac{W_{N\>}[\>u_1\lc u_{j-1},\ut_i,u_{j+1}\lc u_N\+](x)}
{W_{N\>}[\>u_1\lc u_N\+](x)}\;,\qquad i,j=1\lc N\,.\kern-2em
\vv.3>
\ee
Similarly to the proof of Lemma \ref{Duf}, $c_{ij}(x)$ are quasi\+-constants
of the form \>$c_{ij}(x)=x^{\+\la_i-\la_j}\+r_{ij}(x)$ for some rational
functions \>$r_{ij}(x)$. Since $\lab$ is generic, that is,
\>$q^{\>2(\la_i-\la_j)}\<\ne q^{\>2s}$ for \>$i\ne j$ and any \>$s\in\Z$,
we get that \>$c_{ij}=0$ \>for \>$i\ne j$, and \>$\ut_i(x)=r_{ii}(x)\,u_i(x)$,
\>$i=1\lc N\?$, where \>$r_{ii}(x)$ are rational functions and quasi\+-constants.

\vsk.2>
If \>$q$ is not a root of unity, then the only rational functions that are
quasi\+-constants are constant functions.
\end{proof}

\begin{proof}[Proof of Proposition \ref{DUgen}]
Let \>$\ut_i=\bra\+u_i\ket$, \,$i=1\lc N\?$. Then \>$\Ute=(\ut_1\lc\ut_N)$
is a collection of quasi\+-polynomials of type $\lab$. By Lemmas \ref{top} and
\ref{DU}, \>$\DUt=\DU$. By the proof of Proposition \ref{DUtgen} there are
rational functions \>$r_{ii}(x)$ such that \>$\ut_i(x)=r_{ii}(x)\,u_i(x)$,
\>$i=1\lc N\?$. Since the quasi\+-polynomials \>$\ut_1\lc\ut_N$ are \+log\+-free,
the quasi\+-polynomials \>$u_1\lc u_N$ are \+log\+-free as well.
\end{proof}

\appendix

\section{The Wronskian identities}
\label{s-Wid}

\begin{lem}
\label{common}
Given functions \>$f_1(x)\lc f_k(x)$ and $g(x)$, we have
\vvn.2>
\be
W_k[\>gf_1\lc gf_k\+](x)\,=\,
W_k[\>f_1\lc f_k\+](x)\;\prod_{i=0}^{k-1}\,g(x\+q^{-2\+i})\,.
\ee
\end{lem}
\vsk-.3>
\nt
\vv.5>
The proof is straightforward.

\begin{lem}
\label{Wid}
Given functions \>$f_1(x)\lc f_j(x)$, $g_1(x)\lc g_k(x)$, \>let
\vvn.3>
\be
h_i(x)\,=\,W_{j+1}[\>g_i,f_1\lc f_j\+](x)\,,\qquad i=1\lc k\,.\kern-2em
\vv-.4>
\ee
Then
\vvn-.3>
\be
W_k[\>h_1\lc h_k\+](x)\,=\,W_{j+k}[\>g_1\lc g_k, f_1\lc f_j\+](x)\;
\prod_{l=1}^{k-1}\,W_j[\>f_1\lc f_j\+](x\+q^{-2\+l})\,.
\ee
\end{lem}
\vsk-.3>
\nt
\vv.5>
The proof is similar to that of Lemma 9.4 in \cite{MV2}.

\begin{lem}
\label{Wid2}
Given functions $f_1(x)\lc f_k(x)$, \>we have
\vvn.4>
\be
\sum_{i=1}^k\,(-1)^i\,W_{k-1}[\>f_1\lc f_{i-1},f_{i+1}\lc f_k\+](x)\;
f_i(x\+q^{-2\+l})\,=\,0\,,\qquad l=0\lc k-2\,,
\ee
\be
\sum_{i=1}^k\,(-1)^{k-i}\,W_{k-1}[\>f_1\lc f_{i-1},f_{i+1}\lc f_k\+](x)\;
f_i(x\+q^{\>2-2\+k})\,=\,W_k[\>f_1\lc f_k\+](x)\,.
\ee
\end{lem}
\begin{proof}
Consider a \>$k\,{\x}\,k$ matrix \>$M\<$ with entries
\>$M_{ij}=f_i(x\+q^{-2(j-1)})$ for \>$j<k$ and \>$M_{\il}=f_i(x\+q^{-2\+l})\<$.
If \,$l=0\lc k-2$, two rows of \>$M\<$ are the same, hence \,$\det\>M=0$.
If \,$l=k-1$, then \,$\det\>M=W_k[\>f_1\lc f_k\+]$.
Expanding the determinant in the last row yields the claim.
\end{proof}

\begin{lem}
\label{Wid3}
Given functions $f_1(x)\lc f_k(x)$, \>let
\vvn.3>
\be
g_i(x)\,=\,W_{k-1}[\>f_1\lc f_{i-1},f_{i+1}\lc f_k\+](x)\,,\qquad i=1\lc k\,.
\kern-2em
\vv-.3>
\ee
Then
\vvn-.2>
\be
W_j[\>g_j\lc g_1\+](x)\,=\,W_{k-j}[\>f_{j+1}\lc f_k\+](x\+q^{\>2-2\+j})\;
\prod_{l=0}^{k-2}\,W_s[f_1\lc f_s](x\+q^{-2\+l})\,.
\ee
\end{lem}
\nt
The proof is similar to that of Lemma 9.5 in \cite{MV2}.

\vsk>

\end{document}